\documentclass[12pt]{amsart}

\usepackage{amsfonts,graphicx,color}
\usepackage{amsmath, amsthm, amssymb}

\usepackage{mathbbol}
\usepackage{txfonts}
\usepackage{calc}

\newtheorem{thm}{Theorem}[section]
\newcommand{\bt}{\begin{thm}}
\newcommand{\et}{\end{thm}}

\newtheorem{ex}[thm]{Example}

\newtheorem{cor}[thm]{Corollary}   
\newcommand{\bc}{\begin{cor}}
\newcommand{\ec}{\end{cor}}

\newtheorem{lem}[thm]{Lemma}   
\newcommand{\bl}{\begin{lem}}
\newcommand{\el}{\end{lem}}

\newtheorem{prop}[thm]{Proposition}
\newcommand{\bp}{\begin{prop}}
\newcommand{\ep}{\end{prop}}

\newtheorem{defn}[thm]{Definition}
\newcommand{\bd}{\begin{defn}}    
\newcommand{\ed}{\end{defn}}

\newtheorem{rmrk}[thm]{Remark}   

\newcommand{\br}{\begin{rmrk}}
\newcommand{\er}{\end{rmrk}}

\newtheorem{quest}[thm]{Question}

\newtheorem{example}[thm]{Example}

\newcommand{\GHto}{\stackrel { \textrm{GH}}{\longrightarrow} }
\newcommand{\Fto}{\stackrel {\mathcal{F}}{\longrightarrow} }

\newcommand{\be}{\begin{equation}}

 \newcommand{\ee}{\end{equation}}

\newcommand{\E}{\mathbb{E}}

\newcommand{\diam}{\operatorname{Diam}}

\newcommand{\set}{{\rm{set}}}

\newcommand{\mass}{{\mathbf M}}

\newcommand{\intcurr}{{\mathbf I}}      

\newcommand{\area}{\operatorname{Area}}
\newcommand{\vol}{\operatorname{Vol}}

\newcommand{\CovSpec}{{\operatorname{CovSpec}}}

\newcommand{\rstr}{\:\mbox{\rule{0.1ex}{1.2ex}\rule{1.1ex}{0.1ex}}\:}

\begin{document}

\title{Intrinsic Flat Convergence of Covering Spaces}

\author{Zahra Sinaei}
\thanks{{The first author's research was funded by Swiss NSF $P2ELP2\underline~148909$}}
\address{Courant Institute, NYU}
\email{sinaei@cims.nyu.edu}

\author{Christina Sormani}
\thanks{{The second author's research was funded in part by NSF-DMS-10060059 and a PSC-CUNY grant.}}
\address{CUNY Graduate Center and Lehman College}
\email{sormanic@member.ams.org}


\begin{abstract}
We examine the limits of covering spaces and the covering spectra
of oriented Riemannian manifolds, $M_j$, which converge to a nonzero integral current space, $M_\infty$,  in the intrinsic flat sense. 
We provide examples demonstrating that the covering spaces
and covering spectra need not converge in this setting.   In fact
we provide a sequence of simply connected $M_j$ diffeomorphic
to $\mathbb{S}^4$ 
that converge in the
intrinsic flat sense to a torus $\mathbb{S}^1\times\mathbb{S}^3$.  Nevertheless, 
we prove that if the $\delta$-covers, 
$\tilde{M}_j^\delta$, have
finite order $N$, then a subsequence of the $\tilde{M}_j^\delta$
converge in the intrinsic flat sense to a metric space, 
$M^\delta_\infty$,
which is the disjoint union of covering spaces of $M_\infty$.
\end{abstract}

\maketitle

\section{Introduction}
When sequences of Riemannian manifolds converge smoothly to a limit manifold, their covering spaces converge to covering spaces and their universal covers converge to universal covers.   When one weakens the notion of convergence, such behavior is no longer true.
In fact, one may observe that sequences of increasingly thin flat tori, $M_j=\mathbb{S}^1_{1/j}\times \mathbb{S}^1$, converge in the Gromov-Hausorff sense
to a circle $M_\infty=\mathbb{S}^1$.    Yet their universal covers are Euclidean planes, $\tilde{M}_j=\E^2$.  So the limit of the universal covers
is $\E^2$ which is not a covering space for $\mathbb{S}^1$.    There are also sequences of manifolds $M_j$ which
converge in the Gromov-Hausdorff sense to $M_\infty$ isometric to the Hawaii ring, which has
no universal cover.   

In \cite{SorWei1}, the second author and Guofang Wei defined a notion called a $\delta$-cover, denoted 
$\tilde{M}^\delta$, where the covering maps $\pi:\tilde{ M}^\delta\to M$
are isometries on balls of radius $\delta$.   They proved that when $M_j\GHto M_\infty$ then a subsequence of $\tilde{M}_j^\delta$ converge to a covering space of $M_\infty$.   They applied this notion to prove that the
Gromov-Hausdorff limits of sequences of manifolds with nonnegative Ricci curvature have universal covers.   In \cite{SorWei3}, the notion
of a $\delta$-cover was applied to define the covering spectrum,
$\CovSpec(M)$,
of a Riemannian manifold.  They prove that the covering spectrum is
continuous under Gromov-Hausdorff convergence: 
\be\label{GHCovSpec}
\,\,\,M_j \GHto M_\infty\quad \implies \quad \CovSpec(M_j)\cup \{0\} \to \CovSpec(M_\infty)\cup\{0\}.
\ee    
Additional work on the covering spectrum and related notions
has been conducted by Bart DeSmit, John Ennis, Ruth Gornet, Conrad Plaut, Craig Sutton, Jay Wilkins and Will Wylie \cite{deSmit-Gornet-Sutton}, \cite{deSmit-Gornet-Sutton-2}, \cite{Ennis-Wei}, \cite{Plaut-Wilkins}, \cite{Wilkins}, \cite{Wylie}.

Here we consider sequences of compact oriented Riemannian manifolds, $M_j$, of constant dimension and bounded volume and diameter:
\be\label{DVm}
\vol(M_j)\le V_0,\,\,\, \diam(M_j)\le D_0, \,\,\, \dim(M_j)=m
\ee
which converge in the intrinsic flat sense $M_j \Fto M_\infty$. 
Intrinsic flat convergence was first defined by the second author and Stefan Wenger in \cite{SorWen2} applying work of Ambrosio-Kirchheim \cite{AK}.   They prove that limit spaces obtained under intrinsic flat convergence are either countably 
$\mathcal{H}^m$ rectifiable metric spaces of the same dimension
as the original sequence or the sequence disappears converging to
the $\bf{0}$ limit space.   For example the sequence of collapsing tori,
$M_j=\mathbb{S}^1\times \mathbb{S}^1_{1/j}$, converge in the intrinsic flat sense to
the $\bf{0}$ space.   
The sequence of manifolds which are spheres with increasingly many increasingly thin splines converges in the intrinsic flat sense to a 
sphere (so that only the splines disappear in the limit) \cite{SorWen2}.   Stefan Wenger proved that sequences of oriented manifolds satisfying (\ref{DVm}) have a subsequence which converges in the intrinsic flat sense \cite{Wenger-compactness}.   

In Example ~\ref{2-spheres}, we 
construct a sequence of oriented manifolds, $M_j$, satisfying (\ref{DVm}) diffeomorphic to $\mathbb{RP}^3\times \mathbb{S}^2$ such that
$M_j$ converge in the intrinsic flat sense to $M_\infty$ where $M_\infty$ is simply connected.  The
$\delta$-covers, $\tilde{M}_j^\delta$, also converge in the intrinsic flat sense but to a limit space which is a disjoint pair of spaces
isometric to $M_\infty$.    In this example
$\tfrac{\pi}{2}\in \CovSpec(M_j)$ but $\CovSpec(M_\infty)=\emptyset$
.   

Nevertheless we can prove the following

\begin{thm} \label{Finite}
Suppose $M_j$ are oriented Riemannian manifolds satisfying (\ref{DVm})
and $M_j$ converge in the intrinsic flat sense to connected 
$M_\infty\neq \bf{0}$.   Suppose in addition that $\pi_j: \tilde{M}_j^\delta\to M_j$
are finite covers of order $N$.

Then a subsequence of $\tilde{M}_j^\delta$ converges in the
intrinsic flat sense to $M_\infty^\delta$  which is a metric space
with $N_1$ isometric connected components $\mathring{M}^\delta_\infty$ each of which is a covering space of order $N_2$
over $M_\infty$ where $N_1\cdot N_2=N$.  In addition $\mathring{M}^\delta_\infty$  is covered by the covering
space $\tilde{M}^\delta_\infty$.
\end{thm}

In fact if one takes the isometric product of the sequence in Example~\ref{2-spheres} with a fixed manifold with a finite fundamental group of order $N_2$, we obtain an example of a sequence of manifolds whose $\delta$-covers converge to a pair of disjoint isometric finite order covers of the limit space with $N_1=2$ [See Example~\ref{product}].   As part of the proof of this example
we prove Theorem~\ref{thm-product} which describes the
covering spectra of products of geodesic metric spaces in general.

In Example~\ref{Many more spheres} we construct a sequence of oriented two dimensional manifolds
$M_j$ satisfying (\ref{DVm}) with increasingly many increasingly thin handles which converges in the intrinsic flat sense to $\mathbb{S}^2$ with $\CovSpec(\mathbb{S}^2)=\emptyset$.  In this example, the sequence of $\delta$-covers, $\tilde{M}^\delta_j$ doesn't have any converging subsequence.  These $\delta$ covers are not finite order.
This shows that the assumption of a finite order $\delta$-cover in Theorem \ref{Finite} is crucial.    Note also that
in this example,  there exists $\delta_j \in \CovSpec(M_j)$
such that $\delta_j\to \delta_0>0$ but $\delta_0$ is not in  $\CovSpec(M_\infty)\cup\{0\}$.

In prior work of Wei and the second author \cite{SorWei1}, it is shown that when $M_j \GHto M_\infty$
and $M_j$ are simply connected (so that $\tilde{M}_j^\delta=M_j$)
then $M_\infty=\tilde{M}_\infty^\delta=\tilde{M}_\infty$.   In fact 
the Gromov-Hausdorff limit of $\delta$-covers is a covering of any
$\delta'$-cover of the limit space as long as $\delta'>\delta$.   
This is a key step needed to prove convergence of the covering
spectrum in \cite{SorWei3}.

This is
not true when $M_j \Fto M_\infty$.   In Example~\ref{hHole-appears}
we produce a sequence of oriented 
simply connected four dimensional manifolds $M_j$
satisfying (\ref{DVm})
which have regions $U_j\subset M_j$ isometric to $D^2_\pi\times \mathbb{S}^2_{1/j}$ such that $M_j \setminus U_j$
are not simply connected.   The volumes of the regions $U_j$
converge to $0$ in such a way that they disappear in the limit
forming a hole in the limit space.   In this example $\CovSpec(M_j)=\emptyset$ but $\CovSpec(M_\infty)=\{\pi\}$.

In Example~\ref{2D-hole-appears} we construct a sequence of oriented two dimensional manifolds
$M_j$ satisfying (\ref{DVm}) with increasingly many increasingly dense increasingly small tunnels running
between two regions in a pair of spheres.  
See Figure~\ref{fig-2D-hole-appears} within.
We prove $M_j$ converge in the intrinsic flat sense to a torus, $M_\infty$, due to cancellation of the regions with the increasingly dense tunnels.   We prove 
there exists $\delta_0 \in \CovSpec(M_\infty)$
such that $\delta_0$ is not the limit of any sequence
$\delta_j \in \CovSpec(M_j)$. 

In light of our examples one cannot hope to prove convergence
of the covering spectra of sequences of manifolds as in (\ref{GHCovSpec}), if $M_j$ are only
converging in the intrinsic flat sense to their limit spaces.

This paper begins with a background section reviewing key definitions and theorems related to Gromov-Hausdorff convergence, 
intrinsic flat convergence and pointed intrinsic flat convergence.
We also review $\delta$-covers, the covering spectrum, and 
various new theorems about intrinsic flat convergence.   

In the second section we prove Theorem~\ref{Finite} applying
Wenger's Compactness Theorem of \cite{Wenger-compactness}
and various theorems of the second author appearing in
\cite{Sormani-AA} that have been reviewed in the background
section.

The third section of the paper includes constructive descriptions
and proofs of all the examples along with figures.    We also
state and prove Theorem~\ref{thm-product} in this section.   We close this section with a proof that the sequence of manifolds
$M_j\Fto M_\infty$ in
Example~\ref{Many more spheres} has a sequence of 
$\delta$ covers, $\tilde{M}_j^\delta$, with 
no subsequence that converges in the
pointed intrinsic flat sense.   So one cannot extend to
intrinsic flat convergence the theorem
of the second author and Wei proven in \cite{SorWei3} which states that if $M_j \GHto M_\infty$ compact then $\tilde{M}_j^\delta$ has a subsequence
which converges in the pointed Gromov-Hausdorff sense.

The paper closes with a section of open problems.

\subsection{Acknowledgments}
The first author is grateful to Jeff Cheeger who is serving 
as her mentor at the Courant Institute of Mathematical Sciences.  
{Her postdoc is funded by Swiss National Science Foundation research grant $P2ELP2\underline~148909$. }  The second author's research is supported in part by an individual
NSF Research Grant, NSF DMS-10060059, and a PSC-CUNY grant.
Both authors are grateful to the Mathematical Sciences
Research Institute where they served as a postdoc and
visiting research professor in Fall 2013 supported by MSRI through
NSF Grant No. 0932078 000.
   
\section{Background}
In this section we recall some of the definitions and theorems that will be used in the remainder of the paper.  We first provide some basic background including the definition of a length space and a the
gluing of two metric spaces.
 In Subsection~\ref{Gromov-Hausdorff Convergence} we review the notion of Gromov-Hausdorff convergence,   the Gromov Compactness Theorem and the Gromov Embedding Theorem. The material in this subsection has been gathered from \cite{Gromov, Gromov-poly}. 
 In Subsection~\ref{Intrinsic flat} we first recall the notion of intrinsic flat convergence and continue by reviewing a few results: Theorem~\ref{GH-to-flat} (Gromov-Hausdorff convergence implies intrinsic flat convergence),  Sormani-Wenger Embedding Theorem \cite{SorWen2}, Wenger Compactness Theorem \cite{Wenger-compactness}, and the notion of convergence of points \cite{Sormani-AA}. 
 Subsections~\ref{Delta-Cover} and \ref{CoveringSpectrum} are devoted to a review of  the notions of $\delta$-cover  \cite{SorWei1} and covering spectrum  \cite{SorWei1}. We state theorems which study convergence of $\delta$-covers and the covering spectrum of a sequence of compact length spaces which converge in the  Gromov-Hausdorff sense. See Theorems \ref{delta-cover} \cite{SorWei1} and \ref{CovSpec-convergence} \cite{SorWei1}. In Subsection \ref{ArzelaAscoli} we review Arzela-Ascoli  type theorems both in the Gromov-Hausdorff sense, Theorem~\ref{GH-Arz-Asc} \cite{Grove-Petersen} and the intrinsic flat sense Theorem~\ref{Arz-Asc-Unif-Local-Isom} \cite{Sormani-AA}. We close the background section by reviewing a Bolzano-Weierstrass theorem, Theorem \ref{B-W-BASIC} \cite{Sormani-AA}.

\subsection{Length Spaces}

We first recall the definition of length spaces and geodesic spaces.
See, for example, the textbook of Burago-Burago-Ivanov \cite{BBI}.

\begin{defn}
A  length space is a metric space $(X,d)$ such that every pair of points in the space is joined by a length minimizing rectifiable curve; for every $x,y\in X$
\begin{eqnarray*}
d(x,y)=\inf\{L(C), ~C(0)=x,~C(1)=y\} 
\end{eqnarray*}
where the rectifiable length $L(C)$ obtained as follows,
\be
L(C)=\inf_P\sum_{i=1}^N d(C(t_i),C(t_{i-1}))
\ee
where the infimum is over all partitions $P=\{0=t_0<t_1<...<t_N=1\}$
of arbitrary length.
\end{defn}

\begin{defn}
A metric space is called a geodesic metric space if every
pair of points is joined by a curve, $\gamma$, whose 
rectifiable length, $L(\gamma)$ is
the distance between the points:
\be
d(p,q)=\inf\{ L(C): \,\,C(0)=p,\,\,C(1)=q\}= L(\gamma).
\ee
\end{defn}

We explain now how to glue two length spaces together along a common set (cf. \cite{BBI}).

\begin{defn}
Let $(X_1,d_1)$ and $(X_2,d_2)$ be two length spaces  with
compact subsets $A_i \subset X_i$ and length preserving
surjective maps $f_i:A_i\to A $.
We may glue $X_1$ to $X_2$ along the common set $A$ to
obtain a new metric space, $X_1\sqcup_A X_2$,  which is 
the metric space $(X_1\cup X_2)\slash\sim$  where
$x_1\in X_1$ and $x_2\in X_2$ are equivalent,
 \begin{eqnarray*}
x_1\sim x_2~\text{iff}~ f_1(x_1)=f_2(x_2).
\end{eqnarray*}
$(X_1\cup X_2)\slash\sim$  is equipped with the length metric 
$$
d(p_1,p_2)=\inf\{L(C):~C(0)=p_1,~C(1)=p_2\}
$$
where $L(C)$ is determined by $L_1$ or $L_2$ for segments of $C$ in $X_1$ or $X_2$ respectively. 
\end{defn}

In the next example we see the standard connected sum of
two spheres described by gluing one sphere with a ball removed to
a cylinder and then again to another sphere with a ball removed.   So in this example
we are gluing twice.

\begin{example}\label{metric1}
The metric space formed by gluing
\be
\left(\mathbb{S}^2\backslash B(p_1,\epsilon)\right)
\,\sqcup_{\partial B(p,\epsilon)}\,
\left(\partial B(p,\epsilon)\times [0,L]\right)
\,\sqcup_{\partial B(p,\epsilon)}\,
\left(\mathbb{S}^2\backslash B(p_2,\epsilon)\right)
\ee
is isometric to the metric completion
 of $\mathbb{S}^1\times (0, L+2\pi-2\epsilon)$ endowed
with the metric tensor $g=dr^2+f^2(r)d\theta^2$  where
\begin{align}
f(r)=\left\{\begin{array}{lll}
\sin(r)&r\in(0,\pi-\epsilon]\\
\sin(\epsilon)&  r\in(\pi-\epsilon,L+\pi-\epsilon]\\
\sin(r-\pi+2\epsilon-L)&r\in(L+\pi-\epsilon,L+2\pi-2\epsilon).
\end{array} \right.
\end{align}
Note that the cylinder is of length $L$ and width $\epsilon$ while
the two spheres have fixed size.
\end{example}

\subsection{Gromov-Hausdorff Convergence}\label{Gromov-Hausdorff Convergence}

Gromov defined the Gromov-Hausdorff distance 
between compact metric spaces in \cite{Gromov}.
It is an intrinsic notion built upon the
extrinsic notion of the Hausdorff distance between
subsets $A_j$ in a metric space $Z$:
\be
d_H^Z(A_1, A_2)=\inf\left\{r\,|\,\,A_2\subset T_r(A_1),\,\, A_1 \subset T_r(A_2) \right\}
\ee
where 
\be
T_r(A)=\{z\in Z\,|\, \exists \, a\in A \textrm{ s.t. }d(z, a) <r\}
\ee
is the tubular neighborhood about $A$.

\begin{defn} \label{GH-defn} \cite{Gromov}
Given two compact metric spaces 
$(X_1, d_1)$ and $(X_2, d_2)$ then the Gromov-Hausdorff
distance between the spaces is defined:
\be
d_{GH}((X_1, d_1), (X_2,d_2))
=\inf\left\{d_H^Z(\varphi_1(X_1), \varphi_2(X_2))\,|\,\,\varphi_j: X_j \to Z \right\}
\ee
where the infimum is taken over all common compact metric
spaces $Z$ and all isometric embeddings $\varphi_j: X_j \to Z$:
\be\label{isom-emb}
d_Z(\varphi_j(p), \varphi_j(q))=d_{X_j}(p,q).
\ee
\end{defn}

One says a compact sequence of metric spaces converges in the Gromov-Hausdorff
sense, $(X_j, d_j)\GHto (X_\infty, d_\infty)$,
iff
\be
\lim_{j\to\infty} d_{GH}((X_j, d_j), (X_\infty, d_\infty))=0.    
\ee

In order to prove Gromov-Hausdorff convergence it is
useful to use Gromov's notion of an $\epsilon$-almost isometry:

\begin{defn} \cite{Gromov} 
Let $(X_1, d_1)$ and $(X_2, d_2)$ be two compact metric spaces,
 a map $\phi:X_1 \rightarrow X_2$ is said to be an $\epsilon$-almost isometry, if the following two conditions satisfied
 \begin{eqnarray}
&X_2 \subset T_{\epsilon}(\phi(X_1)),\label{epsilon-isometry1}\\
&|d_2(\phi(x_1),\phi(x_2))-d_1(x_1,x_2)|<\epsilon, ~\text{ for all }~x_1,x_2\in X_1.\label{epsilon-isometry2}
\end{eqnarray}
\end{defn}

\begin{thm}\cite{Gromov}
For compact metric spaces,
 $(X_j, d_j)\GHto (X_\infty, d_\infty)$ iff there exist $\epsilon_j$-almost isometries, 
$\psi_j: X_j \to X_\infty$ with $\epsilon_j\to 0$.
\end{thm}

\begin{example} \label{metric2}Let 
\be
X_j=\left(\mathbb{S}^2\backslash B(p_1,\epsilon_j)\right)\,\sqcup_{\partial B(p,\epsilon_j)}\,
\left(\partial B(p,\epsilon_j)\times [0,L]\right)
\,\sqcup_{\partial B(p,\epsilon_j)}\,
\left(\mathbb{S}^2\backslash B(p_2,\epsilon_j)\right)
\ee
which are isometric to the metric completion of
$\mathbb{S}^1\times (0, L_j)$ where $L_j=L+2\pi-2\epsilon_j$
with the metric  $dr^2+f^2(r)d\theta^2$  where
\begin{align}
f_j(r)=\left\{\begin{array}{lll}
\sin(r)&r\in(0,\pi-\epsilon_j]\\
\sin(\epsilon_j)&  r\in(\pi-\epsilon_j,L+\pi-\epsilon_j]\\
\sin(r-(L+\pi-2\epsilon_j))&r\in(L+\pi-\epsilon_j,L+2\pi-2\epsilon_j).
\end{array} \right.
\end{align}
Observe that the $f_j(r/L_j)$ converge in $C^0[0,1]$ sense as $\epsilon_j\to 0$ to $f_\infty(r/L_\infty)$ where $L_\infty=L+2\pi$ and
\begin{align}
f_\infty(r)=\left\{\begin{array}{lll}
\sin(r)&r\in(0,\pi]\\
0&  r\in(\pi,L+\pi]\\
\sin(r-\pi-L)&r\in(L+\pi,L+2\pi).
\end{array} \right.
\end{align}
We have $X_j\GHto X_\infty$ where 
$X_\infty=\mathbb{S}^2\sqcup_{\{p_0\}}[0,L]\sqcup_{\{p_1\}} \mathbb{S}^2$.  This can be proven using $\epsilon_j$-almost isometries that map level sets of $r/L_j$ in $X_j$ to corresponding level sets of $r/L_\infty$ in $X_\infty$.
\end{example}

If one has a sequence of complete pointed metric
spaces $(Y_j, d_j, p_j)$ one can define pointed Gromov-Hausdorff
convergence to $(Y_\infty, d_\infty, p_\infty)$ iff
\be
\forall R>0 \,\,\,
\lim_{j\to\infty} 
d_{GH}\left(\left(\bar{B}(p_j,R), d_j\right), \left(\bar{B}(p_\infty,R), d_\infty\right)\right)=0    
\ee
where $\bar{B}(p_j,r)=\{x: \,\, d_j(x,p_j)\le R\}$ is endowed with the
restricted metric $d_j$ from $X_j$.

Gromov then proved the following compactness theorem:

\begin{thm}\label{Gr-compact} \cite{Gromov}
If $(X_j, d_j)$ are a sequence of compact metric spaces
with $\diam(X_j)\le D_0$
and the max number of disjoint balls of radius $r$ that
fit in $X_j$ is uniformly bounded by a common function, $N(r)$,
then a subsequence converges in the Gromov-Hausdorff 
sense to a compact metric space, $(X_\infty, d_\infty)$.
\end{thm}

In \cite{Gromov-poly} Gromov proved the following theorem
sometimes referred to as the Gromov Embedding Theorem:

\begin{thm}\label{Gr-embed} \cite{Gromov-poly}
If $(X_j, d_j)$ converges in the Gromov-Hausdorff sense to
$(X_\infty, d_\infty)$ which is compact then there exists 
a common compact metric space, $Z$, and isometric
embeddings $\varphi_j: X_j \to Z$ such that
\be
\lim_{j\to \infty} d_H^Z(\varphi_j(X_j), \varphi_\infty(X_\infty))=0.
\ee
\end{thm}

Recall that a length space is a geodesic space, if the
distance between any pair of points is achieved by a curve,
Gromov proved in \cite{Gromov} that this geodesic
structure is preserved under Gromov-Hausdorff convergence:

\begin{thm}\label{Gr-geod}\cite{Gromov}
If $(X_j,d_j)$ are geodesic metric spaces and they converge
in the Gromov-Hausdorff sense to $(X_\infty, d_\infty)$ then
$(X_\infty, d_\infty)$ is also a geodesic metric space.
\end{thm}

\subsection{Intrinsic Flat Convergence}\label{Intrinsic flat}

In \cite{SorWen2}, the second author and Wenger defined integral current spaces $(X,d,T)$, as a metric space $(X,d)$ with
an integral current, $T\in \intcurr(\bar{X})$, such that $\set(T)=X$ where $\set(T)$ is the set of positive density for $T$.   Recall that Ambrosio-Kirchheim introduced the notion of an integral current on a 
complete metric space,
$T\in \intcurr(Z)$, and as well as the notions of mass measure, $||T||$, and mass, $\mass(T)$,  and 
push forward, $\varphi_\#T$, and boundary,
$\partial T$, in \cite{AK}.   In \cite{SorWen2}, the second author and Wenger apply these notions of Ambrosio-Kirchheim to define the intrinsic flat distance as follows:

\begin{defn} \label{IF-defn} \cite{SorWen2}
Given two $m$ dimensional precompact integral current spaces 
$M_1=(X_1, d_1, T_1)$ and $M_2=(X_2, d_2,T_2)$ then the 
intrinsic flat 
distance between the spaces is defined:
\be
d_{\mathcal{F}}(M_1, M_2)=\inf\left\{d_F^Z(\varphi_{1\#}T_1, \varphi_{2\#}T_2):\,\,\varphi_j: X_j \to Z \right\}
\ee
where the infimum is taken over all common complete metric
spaces $Z$ and all isometric embeddings $\varphi_j: X_j \to Z$
satisfying (\ref{isom-emb}).   Note that if $M_1=\bf{0}$ then the
intrinsic flat distance is defined by setting $\varphi_{1\#}T_1=0$. 
\end{defn}

Wenger and the second author prove $d_{\mathcal F}$ is a metric on the space of precompact
$m$ dimensional
integral current spaces.
In particular $d_{\mathcal F}$  defines a distance between pairs of
oriented Riemannian manifolds with finite volume.

By constructing a common metric space $Z$ where one can isometrically embed two
nondiffeomorphic  oriented  Riemannian manifolds with diffeomorphic subdomains, the second author and Lakzian provide an upper bound for  the intrinsic flat distance between these two integral current spaces:

\begin{thm}\cite{Lakzian-Sormani}\label{diffeomorphic}
Suppose ${\bf{M_1}}=(M_1,g_1)$ and ${\bf{M_2}}=(M_2,g_2)$ are oriented precompact Riemannian
 manifolds with diffeomorphic subregions $U_i\subset M_i$ and diffeomorphism $\psi_i:U\to U_i$ such that
 \be\label{Three}
 {\psi_1}^*g_1(V,V)<(1+\epsilon)^2{ \psi_2}^* g_2(V,V)\quad \forall V\in TU,
 \ee
 and 
  \be\label{Four}
 {\psi_2}^*g_2(V,V)<(1+\epsilon)^2{ \psi_1}^* g_1(V,V)\quad \forall V\in TU.
 \ee
 Taking the intrinsic diameters 
 \be
D_{U_i}=\sup\{\diam_{M_i}(W):~\textrm{$W$ is a connected component of $U_i$}\}\leq \diam(M_i),
\ee
we define a hemispherical width,
\be\label{a}
a>\tfrac{\arccos(1+\epsilon)^{-1}}{\pi}\max\{D_{U_1},D_{U_2}\}.
\ee
Taking the difference in distances with respect to the outside manifolds, 
\be\label{lambda}
\lambda=\sup_{x,y\in U}|d_{M_1}(\psi_1(x),\psi_1(y))-d_{M_2}(\psi_2(x),\psi_2(y))|,
\ee
we define heights,
\be\label{h}
h=\sqrt{\lambda(\max\{D_{U_1},D_{U_2}\}+\tfrac{\lambda}{4})}
\ee
and
\be\label{barbarh}
\bar{h}=\max \{h,\sqrt{\epsilon^2+2\epsilon}D_{U_1},\sqrt{\epsilon^2+2\epsilon}D_{U_2}\}.
\ee
Then the intrinsic flat distance between settled completions is bounded,
\begin{eqnarray}
d_{\mathcal{F}}(M_1,M_2)\leq  &(2\bar{h}+a)\left(\vol_m(U_1)+\vol_m(U_2)+\vol_{m-1}(\partial U_1)
+\vol_{m-1}(\partial U_2)\right)\nonumber\\
&+\vol_m(M_1\backslash U_1)+\vol_m(M_2\backslash U_2).\label{InequalitySL}
\end{eqnarray}
Here the settled completion is just the collection of points
in the metric completion with positive density.
\end{thm}

The second author and Wenger prove 
the following embedding theorem into a common complete metric space $Z$ [see \cite{SorWen2} Theorem 3.2 and Theorem 4.2]:

\begin{thm}\label{converge}\cite{SorWen2}
If a sequence of 
integral current spaces,
$M_{j}=\left(X_j, d_j, T_j\right)$, converges  in the intrinsic flat sense to an 
integral current space,
 $M_0=\left(X_0,d_0,T_0\right)$, then
there is a separable
complete metric space, $Z$, and isometric embeddings  $\varphi_j: X_j \to Z$ such that
$\varphi_{j\#}T_j$ flat converges to $\varphi_{0\#} T_0$ in $Z$
and thus converge weakly as well. 
\end{thm}

 Combining Gromov's Embedding Theorem with Ambrosio-Kirchheim's Compactness Theorem the second author and Wenger prove another compactness theorem.

\begin{thm} \cite{SorWen2}\label{GH-to-flat}
Given a sequence of $m$ dimensional integral current spaces $M_j=\left(X_j, d_j, T_j\right)$ such that 
$
\left(\bar{X}_{j}, d_{j}\right) \GHto \left(Y,d_Y\right)$,
where $(Y, d_Y)$ is compact, and
\be
\mass(M_j)\le V_0 \,\,\,\textrm{ and } \,\,\, \mass(\partial M_j)\le A_0
\ee
then a subsequence converges  in the 
intrinsic flat sense 
\be
\left(X_{j_i}, d_{j_i}, T_{j_i}\right) \Fto \left(X,d_X,T\right)
\ee
where either $\left(X,d_X,T\right)$ is the ${\bf 0}$ current space
or $\left(X,d_X,T\right)$ is an $m$ dimensional integral current space
with $X \subset Y$ with the restricted metric $d_X=d_Y$.
\end{thm}

Since intrinsic flat limits are always the same dimension as the sequence, the limits are always $\bf{0}$ when the GH limit has
collapsed to a lower dimension.   It is also possible that the
limit is not zero but that $X$ is a strict subset of $Y$ because
some points have disappeared in the limit.

In \cite{Sormani-AA}, the second author clarifies the meaning of disappearance using the embeddings from Theorem~\ref{converge}.
We say $p_j\in X_j$ converges to $p\in X_\infty$ if
\be\label{convergence-point}
\lim_{j\to\infty} \varphi_j(p_j)=\varphi_\infty(p) \in Z
\ee
and we say $p_j$ disappears if
\be\label{disappearing-point}
\lim_{j\to\infty} \varphi_j(p_j)=z \in Z
\ee
but $z\notin \varphi_\infty(X_\infty)$.

In \cite{Sormani-AA} Lemma 4.1, it is proven 
that for almost every $r>0$, $S(p,r)=(\bar{B}(p,r), d_M, T\rstr \bar{B}(p,r))$ is 
an integral current space itself.  Furthermore,
if $p_j$ converges to $p_\infty$ then
\be \label{balls-converge}
S(p_j,r) \Fto S(p_\infty,r)\neq {\bf{0}}
\ee
  and if $p_j$ disappear then
\be\label{balls-disappear}
S(p_j,r) \Fto \bf{0}.
\ee

\begin{example} \label{metric3}
Let $X_j$ be as in Example~\ref{metric2} which converge,
$(X_j, d_j)\GHto (X_\infty, d_\infty)$ where 
$X_\infty=\mathbb{S}^2\sqcup_{\{p_0\}}[0,L]\sqcup_{\{p_1\}} \mathbb{S}^2$.   The $X_j$ have a natural integral current
structure $T_j$ defined by integration over the spheres and
the cylinder between them.   The sequence has 
an intrinsic flat limit
\be
(X_j, d_j, T_j)\Fto (\set(T_\infty), d_\infty, T_\infty)
\ee
where $T_\infty$ is defined by integration over the spheres
in $X_\infty$ and
\be
\set(T_\infty)=X_\infty\setminus (0,L).
\ee     
Here $d_\infty$ is the restricted metric from
$X_\infty$ and we see that the intrinsic flat limit is not a
geodesic space.  This was proven explicitly using the 
definition of intrinsic
flat convergence by the second author in the 
appendix of \cite{SorWen2}.   

Note that one can also prove this example
by applying \cite{Sormani-AA} Lemma 4.1 by examining
which points disappear.   First observe that 
$p_j$ on the cylinders must disappear because
(\ref{balls-converge}) fails to hold for
$r$ sufficiently small:
\be
d_\mathcal{F}(S(p_j,r), {\bf{0}})\le \vol(B(p_j,r))\to 0.
\ee
Meanwhile $p_j$ away from the cylinders don't disappear because
for $r$ sufficiently small $S(p_j,r)$ are all isometric to
a ball in a standard sphere and so we don't have 
(\ref{balls-disappear}).
From this we see that
\be
X_\infty\setminus [0,L] \subset \set(T_\infty) \subset X_\infty\setminus (0,L).
\ee
To include the end points, $p_{1},p_2$, of the
line segment, one can use a density argument:
\be
\liminf_{r\to 0} \frac{\vol(B(p_{i},r)}{r^2} >0
\ee
thus $p_{i}\subset \set(T)=X$ and
$
\set(T_\infty)=X_\infty\setminus (0,L).
$
\end{example}

Wenger proves the following compactness theorem in \cite{Wenger-compactness}.   It is first stated in the language of integral current
spaces in \cite{SorWen2}:

\begin{thm}\label{thm-Wenger-compactness} [Wenger]  \cite{Wenger-compactness}  
Let $m, N, C, D > 0$ and let $M_j=(X_j, d_j, T_j)$
be a sequence of integral current spaces of the same dimension
such that
\be \label{eqn-compact-1}
\mass(T_j)\le V_0 \textrm{ and } \mass(\partial T_j) \le A_0
\ee 
and 
\be \label{eqn-compact-2}
\diam(X_j) \le D
\ee
then there exists an integral current space, $M$,
of the same dimension
(possibly the $\bf{0}$ space) such that
\be
\lim_{j\to\infty}d_{\mathcal{F}}(M_j, M_\infty)=0.
\ee
\end{thm}

In particular, sequences of oriented Riemannian manifolds without boundary
with a uniform upper bound on volume and
on diameter have a subsequence which 
converges in the intrinsic flat sense
to an integral current space.  The possibility that $M_\infty$
might be the $\bf{0}$ space is a disadvantage of this compactness
theorem when compared with Gromov's Compactness Theorem. For example if  $M_j=\mathbb{S}^1_{1/j}\times \mathbb{S}^1$ then $M_j \GHto \mathbb{S}^1$
but $M_j \Fto \bf{0}$.  

\subsection{Pointed Intrinsic Flat Convergence}

Urs Lang and Stefan Wenger extended the notion of an integral
current defined as in Ambrosio-Kirchheim to a locally integral current
\cite{Lang-Wenger}.  They write $T\in \intcurr_{loc,m}(Z)$.   
Unlike integral currents,
locally integral currents need not have finite mass.  They do have 
$||T||(B(p,r))<\infty$ for all balls.   

Thus one can naturally extend the notion
of an integral current space as follows

\begin{defn}\label{defn-locally}
A locally integral current space
$M=(X,d,T)$ is a metric space $(X,d)$ with a locally
integral current structure $T\in \intcurr_{loc,m}(\bar{X})$ and 
$X=\set(T)$.
\end{defn}   

The advantage of
this new notion is that it includes complete noncompact oriented
Riemannian manifolds with infinite volume.    

Almost every ball $S(p,r)=(B(p,r), d, T\rstr B(p,r))$ in 
such a locally integral current
space is an integral current space itself.   In particular
$\mass(S(p,r))$ is finite.  

One can then naturally
define pointed intrinsic flat convergence:

\begin{defn} \label{defn-pointed-IF-conv}
A sequence  of pointed locally integral current
spaces $(M_i, p_i)$ converge in the
pointed intrinsic flat sense to $(M_\infty, p_\infty)$ iff: 
\be\label{pted-flat}
\textrm{for almost every } r>0
\textrm{ we have }
S(p_i, r) \Fto S(p,r).
\ee
\end{defn}

As a consequence of the Lang-Wenger's Pointed Compactness Theorem
in \cite{Lang-Wenger}, one immediately has the following pointed compactness theorem for locally integral current spaces:

\begin{thm}[Lang-Wenger]
If a sequence of pointed locally integral current spaces, 
$(M_j,x_j)=((X_j, d_j, T_j), x_j)$
has for all $r>0$ the uniform bound
\be
\sup_{j\in \mathbb{N}}\left(\mass(T_j \rstr B(x_j, r)) + \mass (\partial T_j \rstr B(x_j,r))\right)<\infty
\ee
then a subsequence converges in the pointed intrinsic flat sense
to a locally integral current space (possibly the $\bf{0}$ space).
\end{thm}


\subsection{$\delta$-Covers}\label{Delta-Cover}
First, we recall the definition of covering space from
Spanier's texbook  \cite{Spanier-text}.

\bd 
We say a connected space $Y$ is a covering space of $X$ if there is a
continuous map $\pi :Y \to X$ such that $\forall x \in X$ there is an open neighborhood 
$U$ such that $\pi^{-1}(U)$ is a disjoint union of open subsets of $Y$ each of which is mapped homeomorphically onto $U$ by $\pi$. 
\ed

Let $Y$ be a locally path connected length space and $\mathcal{U}$ be any open covering of $Y$. For any $p \in Y$, Spanier  \cite{Spanier-text} shows,  there is a covering space, $\tilde{Y}_{\mathcal U}$, of $Y$ with covering group $\pi_1(Y,{\mathcal U},p)$, where $\pi_1(Y,{\mathcal U},p)$ is a normal subgroup of $\pi_1(Y,p)$, generated by homotopy classes of closed paths having a representative of the form $\alpha^{-1}\circ\beta\circ\alpha$, where $\beta$ is a closed path lying in some element of $\mathcal{U}$ and $\alpha$ is a path from $p$ to $\beta(0)$. 

The second author and Wei  define the notion of $\delta$-cover as follows in \cite{SorWei1}:

\bd\cite{SorWei1}\label{defn-delta-cover}
 Given $\delta> 0$, the $\delta$-cover, denoted $\tilde{Y}^{\delta}$, of a length space $Y$, is defined to be $\tilde{Y}_{\mathcal{U}_{\delta}}$ where $\mathcal{U}_{\delta}$ is the open covering of $Y$ consisting of all balls of radius $\delta$.
The covering group will be denoted $\pi_1(Y,\delta, p)\subset\pi_1(Y, p)$ and the group of deck transforms of $\tilde{Y}^{\delta}$ will be denoted $G(Y,\delta)=\pi_1(Y,p)/{\pi_1(Y,\delta,p)}$.
\ed

\begin{ex}\label{ex-RP-delta}
Consider the standard real projective space $Y=\mathbb{RP}^k$.
For $\delta\le \pi/2$,  balls of radius $\delta$ are simply connected and so $\pi_1(Y,\delta,p)$ is trivial.  Thus we have $\tilde Y^\delta$
is the universal cover of $\mathbb{RP}^k$ which is $\mathbb{S}^k$.  For $\delta>\pi/2$, there is at least one ball which is not simply connected, so we have a nontrivial $\pi_1(Y,\delta, p)\subset \pi_1(Y)$.  Since $\pi_1(Y)$ only contains two elements then 
$\pi_1(Y,\delta, p)=\pi_1(Y)$ and so $\tilde{Y}^\delta=Y$.   
\end{ex}

Moreover they prove:
\bl \cite{SorWei1}
 The $\delta$-covers of complete length spaces are monotone
in the sense that if $r < t$, then $\tilde{X}^r$ covers $\tilde{X}^ t$. In fact, $\tilde{X}^ r$ is the $r$-cover
of the complete length space $\tilde{X}^ t$.
\el 

\begin{thm}\label{delta-cover}\cite{SorWei1}
Let $X_j$ be a sequence of compact, connected, locally path-connected length spaces that converge
to $X_\infty$ in Gromov-Hausdorff topology. The $\delta$-covering of $X_j$, $(\tilde{X}^\delta_j,\tilde{p}_j)$, converges in the pointed Gromov-Hausdorff  metric to $({X}_\infty^{\delta},\tilde{p}_{\infty})$, then $({X}_\infty^{\delta},\tilde{p}_{\infty})$ is a covering space of $X_\infty$, which is covered
by the $\delta$-cover of $X_\infty$, $\tilde{X}_\infty^{\delta}$. Furthermore, for all $\delta_2>\delta$, $X_\infty^{\delta}$ covers $\tilde{X}_\infty^{\delta_2}$. So we have
covering projections mapping
\be
\tilde X_\infty^{\delta}\to X_\infty^{\delta}\to\tilde{X}_\infty^{\delta_2}\to X_\infty.
\ee
\end{thm}

They also show that the $\delta$-cover of a converging sequence of compact length spaces has converging subsequence.

\bt\label{convergence-delta}\cite{SorWei3}
If a sequence of compact length space $X_j$ converge to a compact length space $X_\infty$ in the Gromov-Hausdorff topology, then for any $\delta>0$ there is a subsequence of $X_j$ such that their $\delta$-cover also converges  in the pointed Gromov-Hausdorff topology.
\et

\subsection{Covering Spectrum}\label{CoveringSpectrum}
The second author and Wei use the notion of $\delta$-cover to introduce the notion of covering spectrum for  complete length spaces:

\begin{defn}\cite{SorWei3}
Given a complete length space $X$, the covering spectrum
of $X$, denoted $\CovSpec(X)$ is the set of all $\delta > 0$ such that
\be
\tilde{X}^{\delta}\neq\tilde{X}^{\delta'}
\ee
for all $\delta' > \delta$.
\end{defn}

For a compact length space $X$, the $\CovSpec(X)\subset (0, \diam(X))$.   Applying Example~\ref{ex-RP-delta} we see that $\CovSpec(\mathbb{RP}^k)=\{\pi/2\}$.

They also prove 

\begin{thm}\cite{SorWei3}\label{CovSpec-convergence}
If $X_j$ is a sequence of compact length spaces converging to a compact length space $X_\infty$, then for each $\delta\in \CovSpec(X_\infty)$, there is $\delta_j\in \CovSpec(X_j)$ such that $\delta_j\to \delta$. Conversely, if $\delta_j \in \CovSpec( X_j)$ and $\delta_j\to \delta > 0$, then  $\delta \in \CovSpec (X_\infty)$. 
\end{thm}

\subsection{Review of Arzela-Ascoli Theorems}\label{ArzelaAscoli}

In \cite{Grove-Petersen}, Grove and Petersen prove in detail
the following theorem which they attribute to Gromov.  It is
applied in the work of the second author and Wei \cite{SorWei1}
to prove their
results about the Gromov-Hausdorff limits of $\delta$-covers:

\begin{thm}\label{GH-Arz-Asc} \cite{Grove-Petersen}
Suppose $(X_j,d_j)$ and $(X'_j, d'_j)$ are compact metric spaces 
converging in the Gromov-Hausdorff sense to $(X_\infty, d_\infty)$
and $(X'_\infty, d'_\infty)$ respectively.
If $F_j: X_j \to X'_j$ are equicontinuous, then there is a 
subsequence of the $F_j$ converging to a continuous function
$F_\infty: X_\infty\to X'_\infty$.

More precisely, there exists isometric embeddings as in (\ref{isom-emb})
of the subsequence $\varphi_j: X_j \to Z$,
$\varphi'_j: X'_j\to Z'$,
such that 
\be
d_H^Z(\varphi_j(X_j), \varphi_\infty(X_\infty)) \to 0
\textrm{ and }
d_H^{Z'}(\varphi'_j(X'_j), \varphi'_\infty(X'_\infty)) \to 0
\ee
and for any sequence $p_j\in X_j$ converging to $p\in X_\infty$:
\be
\lim_{j\to\infty} \varphi_j(p_j)=\varphi_\infty(p) \in Z
\ee
one has
\be 
\lim_{j\to\infty}\varphi_j'(F_j(p_j))=\varphi_\infty'(F_\infty(p_\infty)) \in Z'.
\ee
\end{thm}

Such a powerful theorem for arbitrary
sequences of equicontinuous functions, $F_j$,
does not hold when the sequences only converge 
in the intrinsic flat sense.  
In \cite{Sormani-AA}, the second author provides
a counter example to the full extension of this
theorem.  Nevertheless she proves
the following theorem which can be applied to sequences
of covering maps of $\delta$-covers of oriented manifolds
that are converging in the intrinsic flat sense:

\begin{thm}\label{Arz-Asc-Unif-Local-Isom} \cite{Sormani-AA}   
Let $M_j=(X_j, d_j, T_j)$
and $M'_j=(X'_j,d'_j,T'_j)$  be integral current spaces 
converging in the intrinsic flat sense to 
$M_\infty=(X_\infty, d_\infty, T_\infty)$
and $M'_\infty=(X'_\infty, d'_\infty, T'_\infty)$ respectively.

Fix $\delta>0$.
Let $F_j: M_j\to M'_j$ be continuous maps which are current preserving isometries
on balls of radius $\delta$ in the sense that:
\be \label{iso-sat} 
\forall x\in X_j, \,\, F_j: \bar{B}(x,\delta) \to \bar{B}(F_j(x),\delta)\textrm{ is an isometry}
\ee
and
\be\label{curpres}
F_{j\#}(T_j\rstr B(x,r))=T'_j\rstr B(F(x),r) \textrm{ for almost every } r\in (0,\delta).
\ee
Then, when $M_\infty\neq {\bf{0}}$, one has $M'_\infty \neq {\bf{0}}$ and
there is a subsequence, also denoted $F_j$, which
converges to a (surjective) local isometry
\be
F_\infty: \bar{X}_\infty \to \bar{X}'_\infty.
\ee
More
specifically, there exists isometric embeddings as in (\ref{isom-emb})
of the subsequence $\varphi_j: X_j\to Z$,
$\varphi'_j: X'_j\to Z'$,
such that 
\be
d_F^Z(\varphi_{j\#} T_j, \varphi_{\infty\#} T_\infty)\to 0 \textrm{ and }
d_F^{Z'}(\varphi'_{j\#} T'_j , \varphi'_{\infty\#} T'_\infty)\to 0
\ee
and for any sequence $p_j\in X_j$ converging to $p\in X_\infty$:
\be\label{point-infty}
\lim_{j\to\infty} \varphi_j(p_j)=\varphi_\infty(p) \in Z
\ee
one has
\be \label{iso-infty}
\lim_{j\to\infty}\varphi_j'(F_j(p_j))=\varphi_\infty'(F_\infty(p_\infty)) \in Z'.
\ee
When $M_\infty={\bf{0}}$ and $F_j$ are surjective, one has
$M'_\infty={\bf{0}}$.
\end{thm}

\subsection{Basic Bolzano-Weierstrass Theorem}
Given  a sequence of compact metric spaces, $X_j\GHto X$ and for $x_j\in X_j$,  Gromov proves that there is a converging subsequence again  denoted by $x_j$, which converges to $x\in X$. This is not true in general  if a sequence converge only in  intrinsic flat sense.  In \cite{Sormani-AA} the second author prove the following Bolzano-Weierstrass theorem.

\begin{thm}\label{B-W-BASIC}\cite{Sormani-AA}
Suppose $M_j=(X_j, d_j, T_j)$ are $m$ dimensional
integral current spaces 
which converge in the intrinsic flat sense to a 
nonzero integral current space 
$M^m_\infty=(X_\infty, d_\infty, T_\infty)$.
Suppose there exists $r_0>0$, a positive function
$h:(0,r_0)\to (0,r_0)$, and a sequence
$p_j \in M_j$ such that for almost every $r\in (0, r_0)$ 
\be \label{got-this}
\liminf_{j\to \infty} d_{\mathcal{F}}(S(p_j,r),0) \ge h(r)>0.
\ee 
Then there exists a subsequence, also denoted $M_j$, such that
$p_{j}$ converges to $p_\infty\in \bar{X}_\infty$.
\end{thm} 

\begin{rmrk}\label{BW-rmrk}
Theorem 5.6 in \cite{SorWen2} states that
smooth (or just Lipschitz) convergence implies intrinsic flat 
convergence.  Thus, if a sequence of balls converges
smoothly then we have (\ref{got-this}) and so we
can apply Theorem~\ref{B-W-BASIC}.
\end{rmrk}

\section{Proof of Theorem~\ref{Finite}}

In this section we prove Theorem~\ref{Finite}.

\begin{proof}
We will prove this theorem in four steps:

\vspace{.1cm}
We claim: {\em 
There exists a subsequence of $\tilde{M}_j^\delta$ which converges in the intrinsic flat sense to $M_\infty^\delta$ and each connected component $\mathring{M}^\delta_\infty$ is a covering space
which is an isometry on balls of radius $\delta$ and thus
$\tilde{M}^\delta_\infty$ covers $\mathring{M}^\delta_\infty$.}

\vspace{.1cm}
Let  $F_j$ denote the fundamental domain of the covering  $\pi_j:\tilde{M}_j^{\delta}\to M_j$ based  at $\tilde{x}_j$ in $\tilde{M}_j^{\delta}$.  The covering maps $\pi_j$ send  $F_j$ onto $M_j$ and the union of   $g_j^iF_j$ cover $\tilde{M}_j^{\delta}$,  where $\{g_j^i\}_{i=1}^N\subset G(M_j,\delta)$.  Since $\diam(F_j)\leq 2D_0$, therefore the  
$\diam \tilde{M}_j^{\delta}$  are uniformly  bounded by $2DN$. The boundary of the $F_j$ has measure zero and therefore the  $\vol(\tilde{M}_j^{\delta})$ are uniformly bounded by $NV_0$. Then by Wenger's Compactness Theorem of \cite{Wenger-compactness} [cf. Theorem \ref{thm-Wenger-compactness}], a subsequence of $\tilde{M}_j^\delta$ converges in the intrinsic flat sense to $M_\infty^\delta$. 
  
With the induced metric from $\pi_j$, the $\tilde{M}_j^\delta$ is an oriented Riemannian manifold and $\pi_j$ is an orientation preserving surjective local isometry such that it is isometry on balls of radius $\delta$.  Thus we may apply the second author's Arzela-Ascoli Theorem of \cite{Sormani-AA} (cf. Theorem~\ref{Arz-Asc-Unif-Local-Isom}).  Since  $M_{\infty}\neq 0$, then $M^\delta_\infty\neq 0$.  Furthermore there is a subsequence such that $\pi_j$ converges to a (surjective) local isometry map $\pi_{\infty}$ and so it is distance decreasing. Therefore the restriction map to each connected component $\pi_{\infty}:\mathring{M}^\delta_\infty\to M_{\infty}$ is a covering map which is an isometry on balls of radius less than $\delta$.   Thus we have our first claim.

\vspace{.1cm}
Our second claim: {\em
For any point $p\in M_{\infty}$, the preimage $\pi_\infty^{-1}(p)$ has at least $N$ distinct points in $M^\delta_\infty$.} 

\vspace{.1cm}
Given $p\in M_\infty$, then by Lemma 3.4 in \cite{Sormani-AA} there is a sequence of points $p_j$ in $M_j$  converging 
to $p$ in the sense of  (\ref{convergence-point}). By (\ref{balls-converge}) we know the $S(p_j,r)$ do not converge in the intrinsic flat sense to $\bf{0}$.  Let  $\tilde{p}_j^1,\ldots,\tilde{p}_j^N$ in $\tilde{M}_j^{\delta}$ denote the lifting of $p$. 

If for some $k$, $\tilde{p}_j^k$ is a disappearing sequence, then by (\ref{balls-disappear}) there exist $\delta_1$ such that for almost every $r\in(0,\delta_1)$, 
 \be
S(\tilde{p}_j^k,r) \Fto \bf{0}.
\ee
The maps $\pi_j:B(\tilde{p}_j^k,r)\to B(p_j,r)$ satisfy the assumptions of Theorem~ \ref{Arz-Asc-Unif-Local-Isom} and also  they are surjective. Therefore
 \be
S({p}_j,r) \Fto \bf{0},
\ee
which is a contradiction.  

By the second author's Bolzano-Weierstrass Theorem of
\cite{Sormani-AA} (cf. Theorem~ \ref{B-W-BASIC}), there exist a subsequence of collection 
of points $\{\tilde{p}_j^1,\ldots,\tilde{p}_j^N\}$ which converges to $\{\tilde{p}_1,\ldots,\tilde{p}_N\}$  in $M_{\infty}^{\delta}$.    Since
\be
\tilde{d}_j(\tilde{p}_j^i, \tilde{p}_j^k)\ge \delta
\ee
where $\tilde{d}_j$ are the metrics on $\tilde{M}^\delta_j$,
we have
\be
\tilde{d}_(\tilde{p}_\infty^i, \tilde{p}_\infty^k)
=\lim_{j\to\infty}\tilde{d}_j(\tilde{p}_j^i, \tilde{p}_j^k)
\ge \delta>0
\ee
where $\tilde{d}$ is the metric on $M^\delta_\infty$. 
So we have $N$ distinct points in the set
$\{\tilde{p}_1,\ldots,\tilde{p}_N\}$  in $M_{\infty}^{\delta}$.

 Moreover, since $\pi_j$ converges to $\pi_\infty$, we have 
\be 
\pi_\infty(\tilde{p}_i)= \lim_{j\to\infty} \pi_j(\tilde{p}_j^i))=\lim_{j\to\infty}p_j=p_\infty,
\ee
where the limit has been taken in the sense of Theorem~\ref{Arz-Asc-Unif-Local-Isom}, see (\ref{point-infty}),  (\ref{iso-infty}). So $\pi_\infty^{-1}$ has at least  $N$ points.   Thus we have our second claim.

\vspace{.1cm}
Our third claim:
{\em Every point in $M_\infty$ has exactly $N$ lifting points in $M^\delta_{\infty}$}.

\vspace{.1cm}
Suppose  $\tilde{p}_1,\ldots,\tilde{p}_{N'}$, $N'>N$, in $M^\delta_\infty$ are distinct lifting points for a point  $p$. By Lemma $3.4$ in \cite{Sormani-AA}, there exist  $\tilde{p}_j^1,\ldots,\tilde{p}_j^{N'}$ in $\tilde{M}^\delta_j$  such that $\tilde{p}_j^i$ converges to $\tilde{p}_i$ and 
\be
\lim_{j\to\infty}\tilde{d}_j(\tilde{p}_j^i,\tilde{p}_j^k)=\tilde{d}(\tilde{p}_i, \tilde{p}_k)\quad \text{for}~i,k=1,\ldots,N'
\ee
where $\tilde{d}_j$, $\tilde{d}$ are the metrics on $\tilde{M}^\delta_j$ and $M^\delta_\infty$. 

The map $\pi_j$ is covering map of order $N< N'$, so there are at least two points $\tilde{p}_j^i$, $\tilde{p}_j^k$ such that $\pi_j(\tilde{p}_j^i)\neq\pi_j(\tilde{p}_j^k)$. Moreover the covering maps $\pi_j$ converge to $\pi_\infty$ and so $\pi_j(\tilde{p}_j^1),\ldots,\pi_j(\tilde{p}_j^{N'})$ converges to $p$ and so 
\be 
\lim_{j\to\infty}d_j(\pi_j(\tilde{p}_j^i),\pi_j(\tilde{p}_j^k))=0,
\ee
where $d_j$ is the metric on $M_j$.
For $j$ big enough, $d_j(\pi_j(\tilde{p}_j^i),\pi_j(\tilde{p}_j^k))<\delta$ and since $\pi_j$ is isometry on balls of radius $\delta$ we have 
\be
\tilde{d}(\tilde{p}_i, \tilde{p}_k)=\lim_{j\to\infty}\tilde{d}_j(\tilde{p}_j^i,\tilde{p}_j^k)=\lim_{j\to\infty}d_j(\pi_j(\tilde{p}_j^i),\pi_j(\tilde{p}_j^k))=0.
\ee
which is a contradiction.   Thus we have our third claim.

\vspace{.1cm}
Our fourth claim: {\em
The connected components of ${M}^\delta_\infty$ are isometric
to each other.
Thus $N=N_1\cdot N_2$ where $N_1$ is the number of 
isometric copies of $\mathring{M}^\delta_\infty$
in ${M}^\delta_\infty$ and $N_2$ is the number of points in
the preimage $\pi_\infty^{-1}(p)$ intersected with each connected
component.} 

\vspace{.1cm}
 Let $\dot{M}^\delta_\infty$ and $\ddot{M}^\delta_\infty$  denote two different connected components of $M^\delta_{\infty}$.   Since the
 limit space $M$ is connected, the fundamental domains are
 connected and so two different connected components of 
 $M^\delta_\infty$ must have distinct copies of the fundamental domain.   Thus there exists  $\dot{p}$ and $\ddot{p}$ in $\dot{M}^\delta_\infty$ and $\ddot{M}^\delta_\infty$ respectively such that 
 $\pi_\infty(\dot{p})=\pi_\infty(\ddot{p})$.    
 
There exist $\dot{p}_j$  and $\ddot{p}_j$ in $\tilde{M}^\delta_j$  which converge to $\dot{ p}$ and $\ddot{p}$ respectively.   There exists
an element $g_j$ in the covering group $G(M_j, \delta)$ which
maps $\dot{p}_j$ closest to $\ddot{p}_j$:
\be\label{new-claim-4}
\tilde{d}_j(g_j \dot{p}_j, \ddot{p}_j) \le 
\min\{ \tilde{d}_j(g \dot{p}_j, \ddot{p}_j): \, g\in G(M_j, \delta)\}
=d_j(\pi_j(\dot{p}_j), \pi_j(\ddot{p}_j)). 
\ee

This defines an isometry $ g_j:\tilde{M}^{\delta}_j\to\tilde{M}^{\delta}_j$.   By the second author's Arzela-Ascoli Theorem of \cite{Sormani-AA} (cf Theorem~\ref{Arz-Asc-Unif-Local-Isom}), there is a subsequence such that $ g_j$ converge to an isometry $g_\infty:M^\delta_{\infty}\to M^\delta_{\infty}$.

By the definition of $g_\infty$ and $\tilde{d}_\infty$, by (\ref{new-claim-4})
and by the definition of $\pi_\infty$ and $d_\infty$ we have

\begin{eqnarray}
d_\infty(g_\infty \dot{p}, \ddot{p})
&=&\lim_{j\to\infty} \tilde{d}_j(g_j \dot{p}_j, \ddot{p}_j)\\
&\le&
\lim_{j\to\infty} d_j(\pi_j(\dot{p}_j), \pi_j(\ddot{p}_j)) \\
&=&d_\infty(\pi_\infty(\dot{p}), \pi_\infty(\ddot{p}))=0.
\end{eqnarray}
Thus $g_\infty \dot{p}= \ddot{p}$.  This
can be extended to the connected components:
\be
g_\infty(\dot{M}^\delta_\infty)=\ddot{M}^\delta_\infty.
\ee
This implies our fourth and final claim.

Thus we have completed the proof of
 Theorem~ \ref{Finite}.
 \end{proof}

\section{Examples}

In this section we present detailed proofs of our 
main examples discussed in the introduction: Example~\ref{2-spheres}, Example~\ref{product}, Example~\ref{hHole-appears},
Example~\ref{2D-hole-appears} and Example~\ref{Many more spheres}. 

\subsection{A Hole Disappears in the Limit}

\begin{example}\label{2-spheres} 
We construct a sequence of oriented manifolds, $M_j$, 
diffeomorphic to $\mathbb{RP}^3\times \mathbb{S}^2$ satisfying (\ref{DVm}) which converges to $M_\infty$ such that $\tfrac{\pi}{2}\in\CovSpec(M_j)$ but $\CovSpec(M_\infty)=\emptyset$ because $M_\infty$ is simply connected.   We prove  $\tilde{M}^\delta_j$ converge in the
intrinsic flat sense to a metric space, $M_\infty^\delta$, which is not a covering space for $M_\infty$.  See Figure~\ref{fig-2-spheres}.
\end{example}

\begin{figure}[h]
\begin{center}
\includegraphics[width=0.7\textwidth]{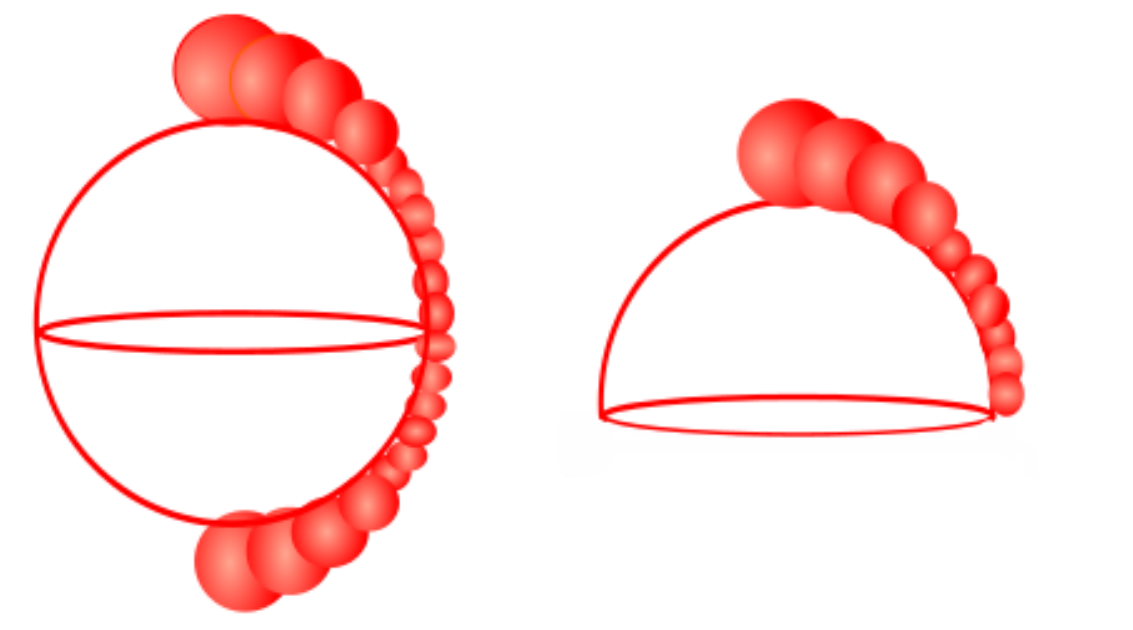}
\caption{$\tilde{M}^\delta_j$ and $M_j$ of Example~\ref{2-spheres}}.
\label{fig-2-spheres}
\end{center}
\end{figure}

\begin{proof}
We define a sequence of Riemannian manifolds $\tilde{M}_j$  as follows, 
\begin{eqnarray*}
(\tilde{M}_j,h_j)=(\mathbb{S}^3\times \mathbb{S}^2,h_j)=\overline{((-\tfrac{\pi}{2},\tfrac{\pi}{2})\times \mathbb{S}^2\times \mathbb{S}^2,dr^2+\cos^2(r)g_{\mathbb{S}^2}+f_j^2(r)g_{\mathbb{S}^2})}
\end{eqnarray*}
where
\begin{align}\label{fjr}
f_j(r)=\left\{\begin{array}{lll}
\tfrac{1}{j}&|r|\in[0,\tfrac{\pi}{4}-\tfrac{1}{j})\\
\text{smoothly monotone}&  |r|\in[\tfrac{\pi}{4}-\tfrac{1}{j},\tfrac{\pi}{4}+\tfrac{1}{j})\\
|\cos(2r)|&  |r|\in[\tfrac{\pi}{4}+\tfrac{1}{j},\tfrac{\pi}{2}). 
\end{array} \right.
\end{align}
By symmetry we have the following isometry
\be
-Id\times Id:=(r,\theta_1,\theta_2)\to(-r,-\theta_1,\theta_2)
\ee
where $(r,\theta_1,\theta_2)\in( -\tfrac{\pi}{2},\tfrac{\pi}{2})\times \mathbb{S}^2\times \mathbb{S}^2$. 

The Riemannian manifolds $M_j$ are defined to be the quotient space 
\be \label{2-spheres-1}
(M_j, \bar{h}_j)=(\mathbb{S}^3\times \mathbb{S}^2,h_j)/\{-Id\times Id,Id\times Id\}
\ee
which is isometric to $(\mathbb{RP}^3\times \mathbb{S}^2,\bar{h}_j)$ where $\bar{h}_j$ is a quotient metric. 

We claim that $\CovSpec(M_j)=\{\pi/2\}$ for every $j$.
Since $\tilde{M}_j$ is diffeomorphic to $\mathbb{S}^3\times \mathbb{S}^2$, it is simply connected.  It is a double cover of $M_j$ and 
\begin{eqnarray*}
\inf\left\{d_{\tilde{M}_j}((r,\theta_1,\theta_2),(-r,-\theta_1,\theta_2)):\, ~(r,\theta_1,\theta_2)\in (-\tfrac{\pi}{2},\tfrac{\pi}{2})\times \mathbb{S}^2\times \mathbb{S}^2\right\}\\ 
\qquad\,\, = \,\, \inf\left\{\,d_{\mathbb{S}^3}((r,\theta_1),(-r,-\theta_1)):\,\, ~(r,\theta_1)\in(-\tfrac{\pi}{2},\tfrac{\pi}{2})\times \mathbb{S}^2\,\right\}\,\,=\,\,\pi
\end{eqnarray*}
because $h_j=g_{\mathbb{S}^3}+f_j^2(r)g_{\mathbb{S}^2}$.  
So by the definition of a $\delta$ cover we have:
\be\label{M-j-delta}
{\tilde{M}_j}^\delta=M_j \textrm{ for }\delta > \tfrac{\pi}{2}.
\textrm{ and }
\tilde{M}_j^\delta=\tilde{M}_j \textrm{ for }\delta \le \tfrac{\pi}{2}.
\ee
By the definition of covering spectrum, we have our claim.

We claim that $\tilde{M}_j$ converges in the Gromov-Hausdorff sense to a metric space 
\be
X=\overline{X}_-\sqcup _{A_-}X_0\sqcup_{A_+} \overline{X}_+
\ee
where
\begin{eqnarray*}
X_-&=&\left((\tfrac{-\pi}{2},\tfrac{-\pi}{4}]\times \mathbb{S}^2\times \mathbb{S}^2,dr^2+\cos^2(r)g_{\mathbb{S}^2}+f_\infty^2(r)g_{\mathbb{S}^2}\right)\\
X_0\,&=&\left([-\tfrac{\pi}{4},\tfrac{\pi}{4}]\times \mathbb{S}^2,dr^2+\cos^2(r)g_{\mathbb{S}^2}\right)
\\
X_+&=&\left([\tfrac{\pi}{4},\tfrac{\pi}{2})\times \mathbb{S}^2\times \mathbb{S}^2,dr^2+\cos^2(r)g_{\mathbb{S}^2}+f_\infty^2(r)g_{\mathbb{S}^2}\right)
\end{eqnarray*}
where $f_\infty(r)=|\cos(2r)|$ and where
\begin{eqnarray}
A_-&=&r^{-1}(\tfrac{-\pi}{4})\textrm{ in $X_0$ and $X_-$} \\
A_+&=&r^{-1}(\tfrac{\pi}{4})\textrm{ in $X_0$ and $X_+$}.
\end{eqnarray}
This follows because $f_j$ converges to $f_\infty$ in $C^1$
and  we have 
\be\label{tilde-I-j}
\epsilon_j \textrm{ almost isometries }\quad
\tilde{I}_j: \tilde{M}_j \to X\quad
\ee  
which take the $r$ level set of $\tilde{M}_j$ to  
the corresponding $r$ levels sets of $X$.  Thus we have our claim,
$
\tilde{M}_j\GHto X.
$ 

We claim the sequence ${M}_j$
converges in the Gromov-Hausdorff sense  to a metric space $Y$, where 
\be
Y=X/\{-Id\times Id,Id\times Id\}
\textrm{ where }(-Id,Id)(r, \theta_1, \theta_2)=(-r, -\theta_1, \theta_2)
\ee
is a fixed point free isometry.
We write 
\be
\pi:X\to Y \textrm{ where } \pi(r, \theta_1, \theta_2)=([(r,  \theta_1)],\theta_2) 
\ee
and $[(r, \theta_1)]$ is the equivalence class containing
both $\pm(r, \theta_1)$.
We have $M_j \GHto Y$ because we have 
\be\label{I-j}
\epsilon_j \textrm{ almost isometries }
{I}_j: {M}_j \to Y
\ee  
which take $([(r, \theta_1)], \theta_2)\in M_j$
to $\pi({\tilde I}_j(r,\theta_1,\theta_2))\in Y$.  

By Gromov's Embedding Theorem there is a compact metric space $Z$ and isometric embeddings $\varphi_j: M_j \to Z$ such that
\be
\lim_{j\to \infty} d_H^Z(\varphi_j(M_j), \varphi_\infty(Y))=0.
\ee
Then a subsequence of $M_j$, we denote it again by $j$,  as integral current spaces  has an intrinsic flat
limit $M_{\infty}$ where $M_{\infty}=\set(T_\infty)\subset Y$
(c.f. Theorem~\ref{GH-to-flat}).

We claim 
\be\label{here-here}
M_\infty=r^{-1}[\tfrac{\pi}{4},\tfrac{\pi}{2}]=\pi(\bar{X}_+)\subset Y.\,\,
\ee
We know by Theorem~\ref{GH-to-flat} that there is a subsequence
$M_j \Fto M_\infty\subset Y$.  We will show any subsequence converges as in (\ref{here-here}).
This will be proven by examining which points disappear in the
limit using \cite{Sormani-AA} Lemma 4.1 (cf. (\ref{balls-converge})
and (\ref{balls-disappear}) ).

For a point $p$ which is in the interior of $\pi(X_0)$, $\mathcal{H}^5(B(p_j,r))\to 0$ and so $S(p_j,r)\Fto \bf{0}$.  Thus we do not
have (\ref{balls-converge}) and so
$p_j$ is a disappearing sequence and $p\notin M_\infty$. 

For any point $p$ in $Y$ there exist $p_j$ in $M_j$ such that $p_j \to p$.
Suppose $p$ lives in the interior of $\pi(\bar{X}_+)$, $p\in W$. 
Since $\pi(X_+)$ is isometric to the $W$
there exists $r>0$ small enough that ${B(p_j,r)}$ with the restricted metric converges smoothly to ${B(p,r)}$.    So it converges in the intrinsic flat sense, i.e. $S(p_j,r)\Fto S(p,r)$ [see Remark~\ref{BW-rmrk}]. Therefore $p_j$ is not a disappearing sequence and $p\in M_\infty$.   

If $p$ is on the boundary of $\pi(X_+)=W$, then it has a positive 
density:
\be
\liminf_{r\to 0} \frac{\mathcal{H}^5(B(p,r)\cap\pi(X_+))}{r^5}>0.
\ee
So $p\in \set(T_\infty)=M_\infty$.

Therefore we have our claim that
$M_\infty$ is isometric to $\pi(\bar{X}_+)$.    
Since all subsequences converge to the same limit space, $M_j \Fto M_\infty$.   In fact we have shown that $M_\infty =\bar{W}$ where
\be\label{p-W}
\forall p\in W \,\,\exists p_j \to p \textrm{ and } r_p>0\textrm{ s.t. } 
B(p_j,r)\to  B(p_\infty,r)) \textrm{ smoothly},
\ee
and
\be\label{p-W-0}
\forall p\in Y\setminus M_\infty 
\,\,\exists p_j \to p \textrm{ and } r_p>0\textrm{ s.t. } 
\mathcal{H}^5(B(p_j,r)) \to 0.
\ee

We claim that $M_\infty$ is simply connected.
Since it is isometric to $\bar{X}_+$, any closed loop has the form: $C(t)=(r(t),\theta_1(t),\theta_2(t))$
where $\theta_i(t)\in \mathbb{S}^2$.  Since $\mathbb{S}^2$ is simply connected,
$C(t)$ is homotopic to 
$
C_1(t)=(r(t),\theta_1(t),\theta_2(0))
$
and homotopic to
$
C_2(t)=(r(t),\theta_1(0),\theta_2(0))
$
which is homotopic to $C_3(t)=(r(0),\theta_1(0),\theta_2(0))$.Thus $\CovSpec(M_\infty)=\emptyset$.

 Since $\tilde{M}_j^\delta\GHto X$, we know a subsequence
of $\tilde{M}_j^\delta \Fto M_\infty^\delta\subset X$.   Imitating the
argument to examine which points disappear, we can show that
\be
M_\infty^\delta=\bar{X}_+\cup \bar{X}_- \subset X
\ee
with the restricted metric.   In fact we can show
$M_\infty^\delta=\bar{W}^\delta$ where 
\be\label{p-W-delta}
\forall p\in W^\delta \,\,\exists p_j \to p \textrm{ and } r_p>0\textrm{ s.t. } 
B(p_j,r)\to B(p_\infty,r))\textrm{ smoothly},
\ee
and
\be\label{p-W-delta-0}
\forall p\in Y\setminus M_\infty 
\,\,\exists p_j \to p \textrm{ and } r_p>0\textrm{ s.t. } 
\mathcal{H}^5(B(p_j,r)) \to 0.
\ee

In particular $M_\infty^\delta$ is not
a connected metric space.  In fact each connected component
of $M_\infty^\delta$ is isometric to $M_\infty$.
\end{proof}

\subsection{Covering Spectra of Products}

\begin{example}\label{product}
We produce a sequence of oriented manifolds, $K_j$,
satisfying (\ref{DVm}) whose $\delta$-cover, $\tilde{K}_j^\delta$, is a finite cover of order $N$.  We prove that 
$\tilde{K}_j^\delta$ converges in the intrinsic flat sense to a metric space ${K}_\infty^\delta$
with $N_1$ connected components and each connected component is a finite cover of $K_\infty$ of order $N_2$ where $N=N_1\cdot N_2$.  In fact $K_j=(M_j,\bar{h}_j)\times (\mathbb{RP}^3,g_{\mathbb{RP}^3})$ 
with the isometric product metric tensor, $\bar{h}_j+g_{\mathbb{RP}^3}$,
where $(M_j,\bar{h}_j)$  is as in (\ref{2-spheres-1}) of Example~\ref{2-spheres} 
and $g_{\mathbb{RP}^3}$ is the standard metric on ${\mathbb{RP}^3}$. 
\end{example}

Before we prove this example, we prove the following theorem:

\begin{thm}\label{thm-product}
If a geodesic metric space
$X=X_1\times X_2$ is endowed with the isometric
product metric 
\be\label{thm-product-1}
d((x_1,x_2),(y_1, y_2))= \sqrt{d_1(x_1,y_1)^2+d_2(x_2,y_2)^2}
\ee
where $d_i$ is the metric on the geodesic metric space $X_i$,
then
\be\label{thm-product-2}
\tilde{X}^\delta= \tilde{X}_1^\delta \times \tilde{X}_2^\delta
\ee
and thus
\be
\CovSpec(X)=\CovSpec(X_1) \cup \CovSpec(X_2).
\ee
\end{thm}

\begin{proof}
Let $\pi_i: \tilde{X}_i^\delta \to X_i$.   Then
\be
\pi: \tilde{X}_1^\delta \times \tilde{X}_2^\delta \to X
\textrm{ defined by } \pi(x_1,x_2)=(\pi_1(x_1), \pi_2(x_2))
\ee
is clearly an isometry on balls of radius $\delta$ because
\be
B((x_1,x_2), \delta) \subset B(x_1,\delta)\times B(x_2,\delta) \subset 
\tilde{X}_1^\delta \times \tilde{X}_2^\delta.
\ee

Now recall in Definition~\ref{defn-delta-cover}
that $\tilde{X}^\delta$ is the covering space of $X$
with covering group $\pi_1(X, \mathcal{U}, p)$ generated
by homotopy classes of closed curves $C$ of the form
$\alpha^{-1}\circ \beta \circ \alpha$ where $\beta$ lies in a ball of radius $\delta$ in $X$.    Since $\pi$ is an isometry on balls of
radius $\delta$, any such curve $C$ in $X$, lifts as a closed 
loop to  $\tilde{X}_1^\delta \times \tilde{X}_2^\delta$.   

To prove
that $\tilde{X}_1^\delta \times \tilde{X}_2^\delta=\tilde{X}^\delta$
we need only show that any other closed curve, not generated as
above, lifts as an open curve to 
$\tilde{X}_1^\delta \times \tilde{X}_2^\delta$.   In other words,
if a closed curve $\gamma$ lifts as an open curve to
$\tilde{X}^\delta$ we need only show that 
$\gamma$ lifts as an open curve to $\tilde{X}_1^\delta \times \tilde{X}_2^\delta$.

Suppose a closed curve $\gamma=(\gamma_1, \gamma_2)$ in
$X$ lifts as an open curve $\tilde{\gamma}$ in $\tilde{X}^\delta$.
We {\em assume on the contrary} that $\gamma$ lifts
to a closed curve $\bar{\gamma}$ in $\tilde{X}_1^\delta \times \tilde{X}_2^\delta$.   Then $\bar{\gamma}=(\bar{\gamma}_1, \bar{\gamma}_2)$ defines a pair of closed curves
$\bar{\gamma}_i$ in $\tilde{X}_i^\delta$ which are lifts of $\gamma_i$
in $X_i$.   

Thus each ${\gamma}_i$ is homotopic in $X_i$ to a product of 
closed curves $C_{i,1}\circ C_{i,2}\circ \cdots \circ C_{i,N}$ in $X_i$ of the form 
\be
C_{i,j}=\alpha_{i,j}^{-1}\circ \beta_{i,j} \circ \alpha_{i,j}
\ee
 where $\beta_{i,j} \subset B(x_{i,j},\delta)\subset X_i$.
 
Let $\sigma_1(t)=(\gamma_1(t), \gamma_2(0))$
and $\sigma_2(t)=(\gamma_1(1), \gamma_2(t))$.   
Then $\gamma(t)$ is homotopic to $\sigma_2\circ \sigma_1$.
Furthermore 
${\sigma}_1$ is homotopic in $X$ to a product of 
closed curves of the form $\alpha^{-1} \circ \beta \circ \alpha$,
where
\be
\alpha(t)=(\alpha_{i,j}(t), \gamma_2(0)) \textrm{ and }
\beta(t)=(\beta_{i,j}(t), \gamma_2(0))
\ee
so that
\be
\beta(t) \subset B(x_{i,j},\delta)\times \{\gamma_2(0)\}
\subset B((x_{i,j},\gamma_2(0)), \delta).
\ee
Thus $\sigma_1$ lifts as a closed curve to $\tilde{X}^\delta$.
Similarly $\sigma_2$ lifts as a closed curve to $\tilde{X}^\delta$.  
Since $\gamma=(\gamma_1,\gamma_2)$ is homotopic to
$\sigma_2\circ \sigma_1$, we see that $\gamma$ lifts
as a closed curve to $\tilde{X}^\delta$.   This is a contradiction.
Thus we have proven (\ref{thm-product-1}).   

Next observe that (\ref{thm-product-1}) implies that
$\tilde{X}^\delta\neq \tilde{X}^{\delta'}$ iff
\be
\tilde{X_1}^\delta\neq \tilde{X_1}^{\delta'}
\textrm{ or } \tilde{X_2}^\delta\neq \tilde{X_2}^{\delta'}.
\ee
So we have (\ref{thm-product-2}) by the definition of
covering spectrum.
\end{proof}

We now prove Example~\ref{product}:

\begin{proof}
Let 
$K_j=(M_j,\bar{h}_j)\times (\mathbb{RP}^3,g_{\mathbb{RP}^3})$ 
with the isometric product metric tensor, $\bar{h}_j+g_{\mathbb{RP}^3}$,
where $(M_j,\bar{h}_j)$  is as in (\ref{2-spheres-1}) of Example~\ref{2-spheres} 
and $(\mathbb{RP}^3,g_{\mathbb{RP}^3})$ is the standard
real projective space.
Note that since we are using the construction in Example~\ref{2-spheres} all notations will be the same as in that example. 

Let  $\tilde{K}_j=(\tilde{M}_j,{h}_j)\times (\mathbb{S}^3,g_{\mathbb{S}^3})$ where $g_{\mathbb{S}^3}$ is the standard metric on ${\mathbb{S}^3}$.  We claim 
\be
\tilde{K}_j^\delta=\tilde{K}_j 
\textrm{ for } 
\delta\le \pi/2
\textrm{ and } 
\tilde{K}_j^\delta=K_j 
\textrm{ for }
\delta>\pi/2.
\ee
This follows from Example~\ref{ex-RP-delta},
 Example~\ref{2-spheres} and
Theorem~\ref{thm-product}.
 
We claim that 
\begin{eqnarray}
K_j=M_j\times \mathbb{RP}^3\GHto Y\times \mathbb{RP}^3,\\
\tilde{K}_j=\tilde{M}_j\times \mathbb{S}^3\GHto X\times \mathbb{S}^3.
\end{eqnarray}
where $M_j \GHto X$ and $\tilde{M}_j \GHto Y$ 
as in Example~\ref{2-spheres}.

The first convergence can be seen using {$\epsilon_j$}-almost isometry $(I_j,Id)$ where $I_j: M_j \to Y$ is
the $\epsilon_j$-almost isometry defined in (\ref{I-j})
in the proof of Example~\ref{2-spheres} and $Id: \mathbb{RP}^3\to\mathbb{RP}^3$ is the identity map. The second convergence can be seen using {$\epsilon_j$}-almost isometry $(\tilde{I}_j,Id)$ where $\tilde{I}_j: \tilde{M}_j \to X$ is
the $\epsilon_j$-almost isometry defined in (\ref{tilde-I-j})
in the proof of Example~\ref{2-spheres} and 
$Id: \mathbb{S}^3\to \mathbb{S}^3$ is the identity map. 

We claim that 
\be
K_j \Fto K_\infty=M_\infty\times \mathbb{RP}^3
\ee
where $M_j \Fto M_\infty$ as in Example~\ref{2-spheres}.

To prove our claim, we first observe that by Theorem~\ref{GH-to-flat} there is a subsequence of $K_j$ which converges
to some $K_\infty \subset Y\times \mathbb{RP}^3$.
For every point $(p,q)\in K_\infty$ there exists  $(p_j,q_j)\in K_{j}$  which converges to $(p,q)$. 

Recall the set $W$ such that $M_\infty=\bar{W}$ described
in (\ref{p-W}) of Example~\ref{2-spheres}.  For 
$p\in W$, $B(p_j,r)$ converges smoothly to $B(p,r)$.   Thus the ball of radius $r$ centered at  $(p_j,q_j)$, $B((p_j,q_j),r)$, converges smoothly to $B((p,q),r)$.  So the points $(p_j,q_j)$
do not disappear as in (\ref{balls-disappear}) and 
so $(p,q)\in K_\infty$. 

If $p$ is in $Y\setminus M_\infty$, then we showed in 
(\ref{p-W-0}) of Example~\ref{2-spheres}  that
for small enough $r>0$, $\mathcal{H}^5(B(p_j,r)) \to 0$
and so 
\be
\mathcal{H}^8(B((p,q),r))
\le \mathcal{H}^8(B((p),r)\times \mathbb{RP}^3)\to 0
\ee 
Thus $(p,q)$ is not in $K_\infty$. 

If $p$ is $M_\infty$, which includes
$p\in \partial W$, then it has positive density, so
there exists $C_1>0$ such that 
\begin{eqnarray*}
\liminf_{r\to 0} \frac{\mathcal{H}^8(B((p,q),r)}{r^8}
&\ge &\liminf_{r\to 0} \frac{\mathcal{H}^8
\left(B\left(p, r/\sqrt{2}\right)\times B\left(q,r/\sqrt{2}\right)\right)}{r^8}\\
&\ge& \liminf_{r\to 0} C_1 \frac{\mathcal{H}^5
\left(B\left(p, r/\sqrt{2}\right)\right)\mathcal{H}^3\left(B\left(q,r/\sqrt{2}\right)\right)}{r^8}\\
&=& \liminf_{r\to 0} C_1\omega_3 2^{3/2} 
\frac{\mathcal{H}^5\left(B\left(p,r/\sqrt{2}\right)\right)}{r^5} 
>0,
\end{eqnarray*}
where 
\be \label{99}
\omega_3=\vol(B(0,1)\subset \E^3)=
\liminf_{r\to 0} \frac{\mathcal{H}^3\left(B(q,r)\right)}{r^3}
\textrm{ for } q\in \mathbb{RP}^3.
\ee
So $(p,q)\in K_\infty$. 

Therefore $K_\infty=M_\infty\times\mathbb{RP}^3$. Since all the subsequences $K_j$ converge to this same $K_\infty$, 
we have $K_j\Fto K_\infty$.  

We claim that {\em 
$\tilde{K}_j$ converges in the intrinsic flat sense 
to a  pair of disjoint $M_\infty \times \mathbb{S}^3$ where $M_j \Fto M_\infty$ as in Example~\ref{2-spheres}.}

To prove our claim, we first observe that by Theorem~\ref{GH-to-flat} there is a subsequence of $\tilde{K}_j$ which converges
to some $\tilde{K}_\infty \subset X\times \mathbb{S}^3$.
For every point $(p,q)\in \tilde{K}_\infty$ there exists  
$(p_j,q_j)\in \tilde{K}_{j}$  which converges to $(p,q)$. 

Recall the set $W^\delta$ such that $M_\infty^\delta=\bar{W^\delta}$ described
in (\ref{p-W-delta}) of Example~\ref{2-spheres}.  For 
$p\in W^\delta$, $B(p_j,r)$ converges smoothly to $B(p,r)$.   Thus the ball of radius $r$ centered at  $(p_j,q_j)$, $B((p_j,q_j),r)$, converges smoothly to $B((p,q),r)$.  So the points $(p_j,q_j)$
do not disappear as in (\ref{balls-disappear}) and 
so $(p,q)\in \tilde{K}_\infty$. 

If $p$ is in $X\setminus M_\infty^\delta$, then we showed in 
(\ref{p-W-delta-0}) of Example~\ref{2-spheres}  that
for small enough $r>0$, $\mathcal{H}^5(B(p_j,r)) \to 0$
and so
\be
\mathcal{H}^8(B((p,q),r))
\le \mathcal{H}^8(B((p),r)\times \mathbb{RP}^3)\to 0
\ee 
Thus $(p,q)$ is not in $\tilde{K}_\infty$. 

If $p$ is in the closure of $M_\infty^\delta$, then it has positive density.   So exactly as in the density argument above
(\ref{99}) we have
\be
\liminf_{r\to 0} \frac{\mathcal{H}^8(B((p,q),r))}{r^8} >0
\ee
because
\be
\liminf_{r\to 0}\frac{\mathcal{H}^5(B(q,r/\sqrt{2}))}{r^5} =\omega_3
\textrm{ for } q\in \mathbb{S}^3.
\ee
So $(p,q)\in K^\delta_\infty$. 

Therefore $\tilde{K}_\infty=M_\infty^\delta\times\mathbb{S}^3$. Since all the subsequences $\tilde{K}_j$ converge to $\tilde{K}_\infty$ we have $\tilde{K}_j\Fto \tilde{K}_\infty$.  

Each $M_\infty\times \mathbb{S}^3$ is a covering space of order 2.  In this Example $N=4$, and $N_1=N_2=2$.
\end{proof}

\subsection{A Hole Appears in the Limit Space}

\begin{example}\label{hHole-appears}
We produce a sequence of four dimensional oriented 
simply connected manifolds $M_j$ 
satisfying (\ref{DVm}) which converge in the
intrinsic flat sense to $M_\infty$ which is diffeomorphic to
$\mathbb{S}^1\times \mathbb{S}^3$.   In particular
$M_j$
which have regions $U_j\subset M_j$ isometric to $D^2\times \mathbb{S}^1_{1/j}$ such that $M_j \setminus U_j$
are not simply connected.   The volumes of the regions $U_j$
converge to $0$ in such a way that they disappear under intrinsic flat convergence forming a hole in the limit space.  In this example $\CovSpec(M_j)=\emptyset$ but $\CovSpec(M_\infty)=\{\pi/2\}$ so we see that
$\CovSpec(M_j)\cup \{0\}$ does not converge to
$\CovSpec(M_\infty)\cup\{0\}$.   See Figure~\ref{hole-appears}.
\end{example}

\begin{figure}[h]
\begin{center}
\includegraphics[width=0.7\textwidth]{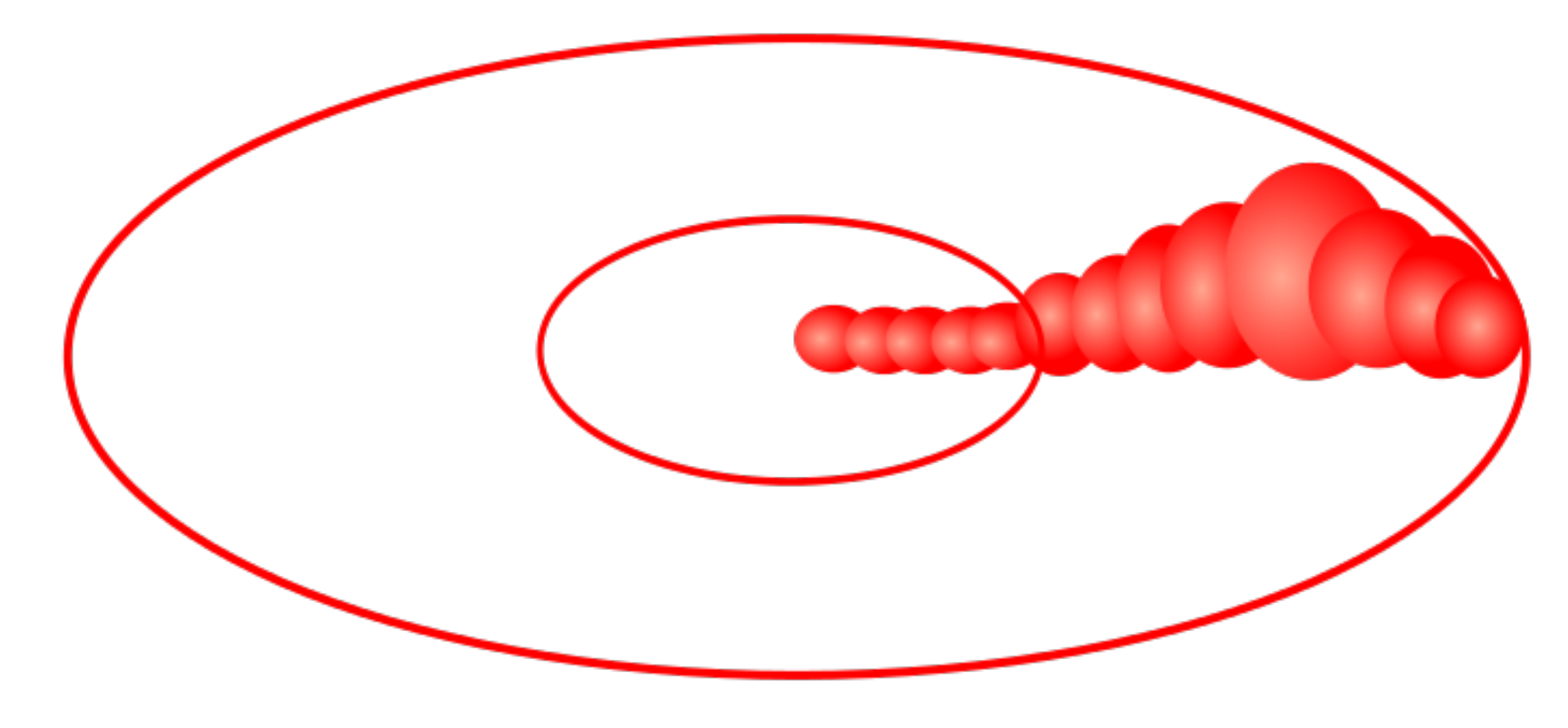}
\caption{Example~\ref{hHole-appears}}
\label{hole-appears}
\end{center}
\end{figure}

\begin{proof}
We consider $[0,1)\times \mathbb{S}^1\times \mathbb{S}^2$
with the Riemannian metric $g_j=dr^2+r^2 d\theta^2+h_j^2(r) g_{\mathbb{S}^2}$ where  
\begin{align}\label{hjr}
h_j(r)=\left\{\begin{array}{lll}
\tfrac{1}{j}&r\in[0,\tfrac{1}{2}-\tfrac{1}{j}]\\
\text{smoothly monotone}& r\in[ \tfrac{1}{2}-\tfrac{1}{j}, \tfrac{1}{2}+\tfrac{1}{j}]\\
|\cos((2r-\tfrac{3}{2})\pi)|&r\in[ \tfrac{1}{2}+\tfrac{1}{j}, 1).
\end{array} \right.
\end{align}

The Riemannian manifold $M_j=\overline{[0,1)\times \mathbb{S}^1\times \mathbb{S}^2}$ is simply connected and so its $\delta$-cover is the same as $M_j$.  
So $\CovSpec(M_j)=\emptyset$.

We claim
$M_j$ converge in Gromov-Hausdorff topology to a metric space 
\be
Y=\overline{D^2_{\frac{1}{2}}}\sqcup_{\mathbb{S}^1_{\frac{1}{2}}} 
\overline{(A_{\frac{1}{2},1}\times \mathbb{S}^2,g)}
\ee 
where   $g=dr^2+r^2d\theta^2+h^2(r)g_{\mathbb{S}^2}$ and 
\be
h(r)=|\cos((2r-\tfrac{3}{2})\pi)| \textrm{ for } r\in [\frac{1}{2},1).
\ee   
This can be seen by
observing that the maps from $M_j$ to $Y$
which preserve the level sets of $r\in[0,1)$
are $\epsilon_j$-almost isometries with $\epsilon_j\to 0$.   The
fact that the maps are not invertible is not a problem.   

We claim that $M_j$ converge in intrinsic flat sense to $M_\infty$, where 
\be
M_\infty=\overline{A_{\tfrac{1}{2},1}\times \mathbb{S}^2}. 
\ee
where   $g=dr^2+r^2d\theta^2+h^2(r)g_{\mathbb{S}^2}$.   Here $r\in [1/2,1]$.  This is proven by examining which points disappear in the
limit using \cite{Sormani-AA} Lemma 4.1 (cf. (\ref{balls-converge})
and (\ref{balls-disappear}) ).
The sets $U_j=r^{-1}[0,1/2)\subset M_j$ clearly
have $\mathcal{H}^4(U_j)\to 0$; so the balls in those sets
cannot satisfy (\ref{balls-converge}).  Thus points in $U_j$
must disappear and 
\be
M_\infty \subset r^{-1}[1/2,1]\subset Y 
\ee
Meanwhile balls in the sets $r^{-1}(1/2,1]=M_j\setminus \bar{U}_j$
do not disappear because they are eventually isometric to the corresponding balls in $M_\infty$.  Finally the points in the
level set $r^{-1}(1/2)$ are verified to have positive density and
thus lie in $M_\infty$.   Thus 
\be
M_\infty = r^{-1}[1/2,1]= \overline{A_{\tfrac{1}{2},1}\times \mathbb{S}^2}\subset Y
\ee

Observe that $M_\infty$ is not simply connected.  It has a 
shortest noncontractible geodesic of length $\pi$ in the
level set $r^{-1}(1/2)$.   Thus 
\be
\frac{\pi}{2} \in\CovSpec(M_\infty).
\ee
In fact $M_\infty$ is diffeomorphic to $\mathbb{S}^1\times \mathbb{S}^3$ because $h(1/2)=h(1)=0$.
\end{proof}

\subsection{Cancellation to a Limit which is a Torus}

Before we present Example~\ref{2D-hole-appears},
we recall Example A.19 of
\cite{SorWen2} in which
 a sequence of manifolds converges in the intrinsic
flat sense to the $\bf{0}$ integral current space 
due to cancellation.

\begin{example}\cite{SorWen2}\label{A.19}
Let $M_j=\partial W_j$ where
\be
W_j=(\mathbb{S}^2\backslash U_j) \times [0,h_j]
\ee
where 
\be
h_j<
\min\left\{\frac{1}{L(\partial U_j)},\frac{1}{j}\right\}.
\ee
and
\be
U_j=\bigcup_{i=1}^{N_j}B(p_i,r_j)
\ee
is the union of balls in $\mathbb{S}^2$ about an increasingly dense
collection of points
such that 
$d(p_i,p_k)>3/j$ and 
\be
\mathbb{S}^2\subset \bigcup_{i=1}^{N_j} B(p_i,{10}/{j})
\ee 
and where $r_j<1/j$. 

So $M_j$ is two copies of $\mathbb{S}^2\backslash U_j$ with opposite orientation glued together by cylinders of the form 
$\partial B(p_i,r_j)\times[0,h_j]$.

Then $M_j \GHto \mathbb{S}^2$ and $M_j \Fto \bf{0}$.
\end{example}

Intuitively Example~\ref{A.19} has $M_j \Fto \bf{0}$ because
sheets of opposite orientation are coming together and causing
cancellation everywhere.  The proof in \cite{SorWen2} provides an explicit
sequence of common metric spaces, $Z_j$, isometric embeddings,
$\varphi_j: M_j \to Z_j$, and integral currents, $B_j$, such
that $\partial B_j=\varphi_{j\#}\lbrack M_j \rbrack$ with
$\mass(B_j)\to 0$.

Our next example does not converge to the $\bf{0}$ integral
current space because we only cause cancellation near the
poles of the spheres:

\begin{example}\label{2D-hole-appears} 
We construct a sequence of oriented two dimensional manifolds
$M_j$ satisfying (\ref{DVm}) with increasingly many tunnels running
between two caps in a pair of spheres.  We prove $M_j$ converge in the intrinsic flat sense to a torus, $M_\infty$.   We
prove that there exists $\delta_0 \in \CovSpec(M_\infty)$
such that $\delta_0$ is not the limit of any sequence
$\delta_j \in \CovSpec(M_j)\cup\{0\}$. 
\end{example}

\begin{figure}[h]
\begin{center}
\includegraphics[width=0.7\textwidth]{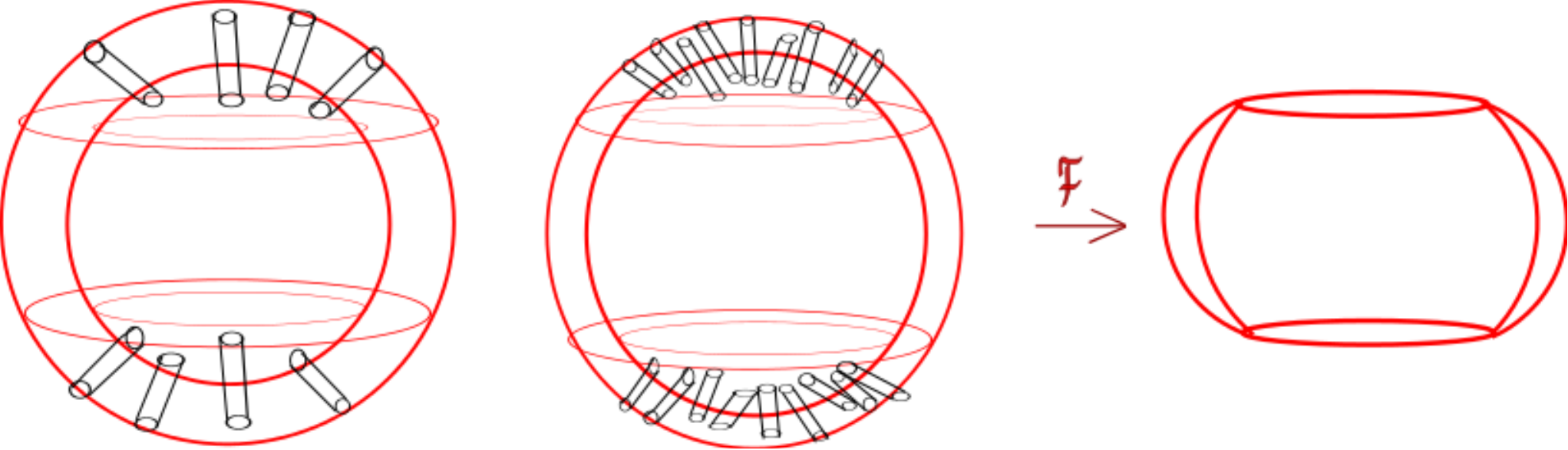}
\caption{Example~\ref{2D-hole-appears}: $M_j \Fto M_\infty$}
\label{fig-2D-hole-appears}
\end{center}
\end{figure}

\begin{proof}
In this example, we use the construction in Example~\ref{A.19}. 
We consider a sphere $\mathbb{S}^2$, with the two poles $P_N,~P_S\in \mathbb{S}^2$ and a two disjoint caps $C_N$,  $C_S$ around them of
radius $\pi/100$.  
We choose a collection of points on  $C_N\cup C_S$,
\begin{align*}
\{p^N_1,p^N_2,\ldots,p^N_{N_j}\}\subset C_N\\
\{p^S_1,p^S_2,\ldots,p^S_{N_j}\}\subset C_S
\end{align*}
such that $d(p_i^N,p_k^N)>3/j$ and $d(p_i^S,p_k^S)>3/j$ for $1\leq i,k\leq N_j$. Moreover we have 
\be
C_N\cup C_S\subset \bigcup_{i=1}^{N_j}\left( B(p_i^N,{10}/{j})\cup B(p_i^S,{10}/{j})\right).
\ee
 Now consider 
\be
W_j=(\mathbb{S}^2\backslash U_j) \times [0,h_j]
\ee
where 
\begin{align*}
U_j=\bigcup_{i=1}^{N_j}\left(B(p_i^N,r_j)\cup B(p_i^S,r_j)\right)~\text{and}~h_j<
\min\left\{\frac{1}{L(\partial U_j)}, \frac{1}{j}\right\}.
\end{align*}
 We choose $r_j$ such that $jr_j\to 0$ as $j\to\infty$.
 
Let $M_j=\partial W_j$ so that $M_j$ is two copies of $\mathbb{S}^2\backslash U_j$ with opposite orientation glued together by cylinders of the form 
$\partial B(p_i,r_j)\times[0,h_j]$, see Figure~\ref{fig-2D-hole-appears}.   Let
$C_N'$ and $C_S'$ be the caps with their cylinders
in $M_j$:
\begin{eqnarray}
C_N' &=& \partial W_j \cap \partial \left( (C_N \backslash U_j)\times [0,h_j] \right)\subset M_j,\\
C_S' &=& \partial W_j \cap \partial \left( (C_S \backslash U_j)\times [0,h_j] \right) \subset M_j.
\end{eqnarray}

We claim  
\be\label{cov-claim}
\CovSpec(M_j) \subset (0, 2h_j+{6}/{j}] \cup (3,\infty).
\ee
Take $\delta>2h_j+6/j$.   Let $C$ be a closed
curve in $M_j$.  If the image of $C$ is homotopic to a curve whose image lies in $C'_N$, then $C$ is generated by loops lying in $C'_N$.
All loops in $C'_N$ are generated either by combinations of 
loops of length
$2\pi r_j< 2h_j+6/j$ that go around single cylinders and  loops of
length $2h_j+6/j$ which go through one cylinder and then a second
cylinder in an adjacent ball.   Thus these loops lift closed to 
$\tilde{M}_j^\delta$.   The same is true if $C$ is homotopic to a curve whose image lies in $C'_S$.   The only closed curves which might lift open to $\tilde{M}_j^\delta$, are ones which 
travel from one cap to another cap and so have length 
$\ge 2(98\pi/100)>6$.   Thus we have
(\ref{cov-claim}).

We claim that
{\em $M_j$ converge in Gromov-Hausdorff topology to a metric space,} 
\be
Y=\mathbb{S}^2\sqcup_{C_N\cup C_S}\mathbb{S}^2,
\ee    
{\em where the spheres have the standard metric and they are glued
together along the caps.   So this is a $2$ dimensional metric
with singularities along $\partial C_N\cup \partial C_S$
that take the form of three half planes meeting along a line.}
\vspace{.1cm}

To prove our claim we construct $\epsilon_j$-almost isometry $\eta_j:M_j\to Y$. 
Every point $p\in M_j$ can be denoted by $(x,t)$ where $x\in \mathbb{S}^2$ and $t\in[0,h_j]$. We define
\begin{align}
\eta_j((x,t))=\left\{\begin{array}{lll}
(p,0)& if~x\in\partial B(p,r_j),~ p=p_i^N~or~p_i^S\\
(x,0)&if~ x\in (C_N\cup C_S)\backslash U_j\\
(x,t)& Otherwise
\end{array} \right.
\end{align}
We will show the map $\eta_j$ satisfies conditions (\ref{epsilon-isometry1}) and (\ref{epsilon-isometry2}) defining an almost isometry by proving 
\begin{eqnarray}
&\left|d(p,q)-d(\eta_j(p),\eta_j(q))\right|<\tfrac{4\pi(\pi-2)}{300}jr_j+2(h_j+\tfrac{20}{j})\label{one}\\
&Y\subset T_{2r_j}(\eta_j(M_j))\label{two}.
\end{eqnarray}

First we prove (\ref{one}). Let $\gamma$ be  a minimizing curve between $\eta_j(p)$ and $\eta_j(q)$ in $Y$.  We construct a curve  $\gamma_j$ between $p$ and $q$ in $M_j$ with  $\eta_j(\gamma_j)\subset\gamma$ and $L(\gamma_j)\geq L(\gamma)$ such that 
\be
\left|L(\gamma)-L(\gamma_j)\right|< \tfrac{4\pi(2\pi-2)}{300}jr_j+2(h_j+\tfrac{20}{j})
\ee
Obviously for such curve we have
\be
\left|d(p,q)-d(\eta_j(p),\eta_j(q))\right|\leq \left|L(\gamma)-L(\gamma_j)\right|.
\ee
For $p=(x,t)$ and $q=(y,t')$ we have the following three cases:
\vspace{.2cm}

Case I. $t=t'=0$ or $t=t'=h_j$.

Case II. $0<t,t'<h_j$ or $t\neq t'$ with $p=(x,0)\in M_j\backslash(C'_N\cup C'_S)$, $q\notin M_j\backslash(C'_N\cup C'_S)$.

Case III. $p=(x,0)\in M_j\backslash(C'_N\cup C'_S)$, $q=(x,h_j)\in M_j\backslash(C'_N\cup C'_S)$.

\vspace{.1cm}
For the Case I, assume $p=(x,0)$ and $q=(y,0)$. Let $\gamma_j(s)=(x_j(s),0)\subset \eta_j^{-1}(\gamma)$ be the shortest path between $p$ and $q$  in $\eta_j^{-1}(\gamma)$. 
From the definition of $\eta_j$ we have
\begin{eqnarray}
L\left(\gamma_j\left|\right._{(C_N\cup C_S)\backslash \partial U_j}\right)&=&L\left(\gamma\left|\right._{(C_N\cup C_S) \backslash U_j}\right)\\
L\left(\gamma_j\left|\right._{M_j\backslash (C'_N \cup C'_S)}\right)&=&L\left(\gamma\left|\right._{Y\backslash (C_N\cup C_S)}\right)\\
\left|L\left(\gamma_j\left|\right._{ \partial U_j}\right)-L\left(\gamma\left|\right._{U_j}\right)\right|&\leq& \tfrac{2\pi/100}{3/j}(2\pi r_j-2r_j)
\end{eqnarray}
since each time $\gamma_j$ goes around a cylinder that
$\gamma$ cuts across, it is
$(2\pi r_j - 2r_j)$ longer, and we know $\gamma$ crosses at
most $2\diam(C_N)/(3/j)= \tfrac{2\pi/100}{3/j}$ cylinders.
Thus we have 
\be
\left|L(\gamma)-L(\gamma_j)\right|< \tfrac{2\pi(2\pi-2)}{300}jr_j.
\ee
\vspace{.1cm}

For the Case II, put $p'=(x,0)$ and $q'=(x,0)$. We have
\begin{eqnarray}
d(p,q)&\leq& d(p',q')+d(p',p)+d(q',q)\\
&\leq &d(p',q')+2(h_j+\tfrac{20}{j}).
\end{eqnarray}
and
\be
d(\eta_j(p),\eta_j(q))=d(\eta_j(p'),\eta_j(q')).
\ee
So we have 
\be
\left|d(p,q)-d(\eta_j(p),\eta_j(q))\right|\leq \left| d(p',q')-d(\eta_j(p'),\eta_j(q'))   \right|+2(h_j+\tfrac{20}{j})
\ee
and by the Case I, 
\be
\left|d(p,q)-d(\eta_j(p),\eta_j(q))\right|\leq \tfrac{4\pi(\pi-2)}{300}jr_j+2(h_j+\tfrac{20}{j}).
\ee

For the Case III, again we assume $\gamma_j\subset \eta_j^{-1}(\gamma)$ to be the shortest path between $p$ and $q$  in $\eta_j^{-1}(\gamma)$. Then we have
\begin{eqnarray}
L\left(\gamma_j\left|\right._{M_j\backslash (C'_N \cup C'_S)}\right)&=&L\left(\gamma\left|\right._{Y\backslash (C_N\cup C_S)}\right)\\
L(\gamma_j)-L(\gamma)&=&L\left(\gamma_j\left|\right._{C'_N\cup C'_S}\right)\leq h_j+\tfrac{20}{j}.
\end{eqnarray}
Therefore we see for all three cases we have (\ref{one}).  For (\ref{two}) we have
\begin{eqnarray}
\eta_j(C_N') &=&\left(C_N\backslash U_j\right) \cup
\{p_1^N,....,p_{N_j}^N\}\subset  Y\\
\eta_j(C_S') &=&\left(C_S\backslash U_j\right) \cup 
\{p_1^S,..., p_{N_j}^S\} \subset  Y\\
\eta_j(M_j\backslash (C_N'\cup C_S'))&=&Y\backslash (C_N\cup C_S).
\end{eqnarray}
 Therefore $T_{2r_j}(\eta_j(M_j))=Y$. Thus we have our claim that $M_j$ converge to $Y$ in the Gromov-Hausdorff sense.
\vspace{.1cm}

We claim that: {\em the intrinsic flat limit $M_\infty$ is the torus:}
\be\label{Mtorus}
M_\infty=(\mathbb{S}^2\backslash {C_N\cup C_S}) \sqcup_{\partial(C_N\cup C_S)}(\mathbb{S}^2\backslash {C_N\cup C_S})= Y \setminus (C_N\cup C_S).
\ee
For any point $p\in Y$ there exist $p_j$ in $M_j$ such that  $p_j\to p$. Suppose $ p\in Y \setminus (C_N\cup C_S)$ and $p$ is not in $\partial({C_N}\cup{C_S})$. There exists $r>0$ small enough, such that $B(p_j,r)$ with restricted metric converges smoothly to  $B(p,r)$. So it converges in intrinsic flat sense, i.e. $S(p_j,r)\Fto S(p,r)$. Therefore $p_j$ is not a disappearing sequence and $p\in M_\infty$. 

For any point  $p$ in the interior of $C_N \cup C_S$, and a sequence $p_j$ converging to $p$,  there exist $r>0$ small enough such that  $B(p_j,r)$
is isometric to a ball in Example~\ref{A.19}.
Since all points in that example disappear, we have
$d_{\mathcal{F}}(S(p_j,r),{\bf{0}})\to 0$.  So $p_j$ is a disappearing sequence, and $p$ is not in $M_\infty$. 
   
If $p$ is on the boundary of $C_N\cup C_S$, then it has a positive 
density:
\be
\liminf_{r\to 0} \frac{\mathcal{H}^2(B(p,r)\cap (Y \setminus (C_N\cup C_S) )}{r^2}>0.
\ee
So $p\in \set(T_\infty)=M_\infty$. Thus we have our claim
in (\ref{Mtorus}).

Let $\delta_0=L(\partial C_N)/2=\pi\sin(\pi/100)<2$.
Since the shortest noncontractible closed geodesic in $M_\infty$
has length $L(\partial C_N)=2\pi \sin(\pi/100)$, we
see that $\delta_0\in \CovSpec(M_\infty)$.
But there are not $\delta_j \in \CovSpec(M_j)$ such that
$\delta_j\to \delta_0$ because 
$\delta_j\in(0,2 h_j+6/j] \cup [3,\infty)$
by (\ref{cov-claim}) and such $\delta_j \to 0$ or $\delta_j \to \delta_\infty\ge 3$.
\end{proof}

\subsection{Converging Sequence with No Converging Subsequence of its $\delta$-Cover}

Recall in \cite{SorWei3}, the second author and Wei proved that
if a sequence $M_j \GHto M_\infty$ then a subsequence of the
$\delta$ covers converges in the pointed Gromov-Hausdorff sense (cf. Theorem~\ref{convergence-delta}).   Here we see this
cannot be extended to the intrinsic flat setting.   Recall the definition
of pointed intrinsic flat convergence of locally integral current spaces given in Definition~\ref{defn-pointed-IF-conv}.

\begin{example}\label{Many more spheres}
We construct a sequence of oriented manifolds, $M_j$, which are spheres with increasingly many
increasingly thin handles, such that $M_j \Fto M_\infty=\mathbb{S}^2$ but 
{$\tilde{M}^\delta$ with 
$\delta=\pi/2$} doesn't have any converging  subsequence in the pointed
intrinsic flat sense. The covering spectra of $M_j$ includes 
$\delta_j\in \CovSpec(M_j)$  with $\delta_j\to \delta_0>0$ where $\delta_0$ is not in  $\CovSpec(M_\infty)\cup\{0\}$.
\end{example}
\begin{figure}[h]
\begin{center}
\includegraphics[width=0.7\textwidth]{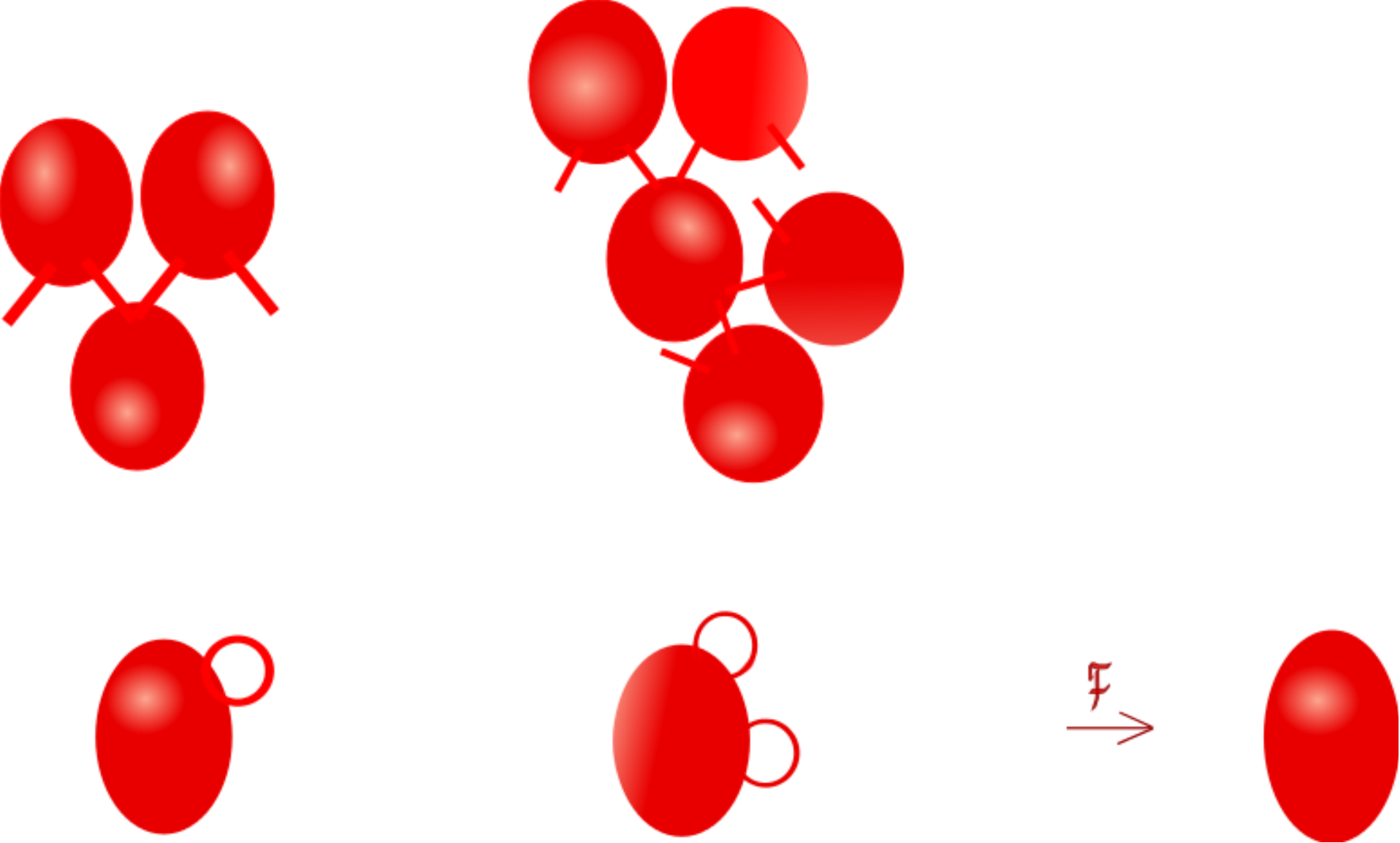}
\caption{Example~\ref{Many more spheres}}
\label{fig-Many more spheres}
\end{center}
\end{figure}

\begin{proof}
For the construction of $M_j$ we use gluing  (cf. Examples~\ref{metric1} and \ref{metric2}). 
Let $B(p_i,\epsilon_i)\subset  \mathbb{S}^2$ be disjoint balls with $\epsilon_i\rightarrow 0$. Let $\epsilon_j^\prime=\min_{i=1}^{j} \{\tfrac{\epsilon_{2i-1}}{10^i},\tfrac{\epsilon_{2i}}{10^i}\}$.   Let 
\begin{eqnarray*}
M_j=W_j\sqcup_{A_1}\left([0,\pi]\times\partial B(q_1,\epsilon_j^\prime)\right)\sqcup_{A_2}\ldots\sqcup_{A_j}\left([0,\pi]\times\partial B(q_j,\epsilon_j^\prime)\right) \\
\end{eqnarray*}
where 
\be
W_j=\left(\mathbb{S}^2\backslash \bigcup_{i=1}^{2j}B(p_i,\epsilon_j^\prime)\right)
\ee
and where
\be
A_i=\partial  B(p_{2i},\epsilon_j^\prime)\cup \partial  B(p_{2i-1},\epsilon_j^\prime)  \textrm{{ for $i=1\ldots j$.  }}
\ee
 which attaches to $\{0\}\times \partial B(q_i,\epsilon_j')$ and $\{\pi\}\times \partial B(q_i,\epsilon_j')$ respectively.  On  $W_j$  we consider  $g_{\mathbb S^2}\left|\right. _{W_j}$. The metric over $[0,\pi]\times\partial B(q_i,\epsilon_j^\prime)$  is defined $dr^2+f_j^2(r)d\theta^2$ where 
\begin{align}
f_j(r)=\left\{\begin{array}{lll}
-\sin(r-\arcsin \epsilon_j)& r\in[0,\eta_j]\\
\textrm{smoothly monotone}&r\in(\eta_j,2\eta_j)\\
\sin(2\eta_j)&r\in[2\eta_j,\pi-2\eta_j]\\
\textrm{smoothly monotone}&r\in(\pi-2\eta_j,\pi-\eta_j)\\
-\sin(r-\arcsin(\epsilon_j))&r\in[\pi-\eta_j,\pi]
\end{array} \right.
\end{align}
where $\eta_j<\arcsin(\epsilon_j)/2$.  This choice of $f_j$ makes a $M_j$ a smooth Riemannian manifold.    

We claim $M_j\Fto \mathbb{S}^2$.  We apply a theorem of  Lakzian  and the second author proven in \cite{Lakzian-Sormani}
(cf. Theorem~ \ref{diffeomorphic})
to prove this claim. Consider  
\be
W_j=\mathbb{S}^2\backslash \bigcup_{i=1}^{2j}B(p_i,\epsilon_j^\prime)\subset (M_j,g_j)
\ee as above. We define  
\be
\bar{W}_j= \mathbb{S}^2\backslash \bigcup_{i=1}^{2j}B(p_i,\epsilon_j^\prime)\subset (\mathbb{S}^2, g_{\mathbb{S}^2}).
\ee
 We consider the induced Riemannian metric from $\mathbb{S}^2$ and $M_j$ on $\bar{W}_j$ and $W_j$: $\left(\bar{W}_j,g_{\mathbb{S}^2}\left|\right._{\bar{W}_j}\right)$  and $\left({W}_j,g_j\left|\right._{{W}_j}\right)$. Applying (\ref{InequalitySL}) we have
 \begin{eqnarray*}\label{flatdistance}
\nonumber d_{\mathcal{F}}(\mathbb{S}^2,M_j)\leq& (2\bar{h}+a)\left(\area(W_j)+\area(\bar{W}_j)+L(\partial W_j)
+L(\partial \bar{W}_j)\right)\\
&+\area(\mathbb{S}^2\backslash \bar{W}_j)+\area(M_j\backslash W_j).
\end{eqnarray*}
where $a$ and $\bar{h}$ are defined in (\ref{a}) and (\ref{barbarh}). To prove our claim we need only to show the right hand side of the equation above converges to zero.
 
The maps $\psi_1=Id: \bar{W}_j\to \bar{W}_j$ and  $\psi_2:\bar{W}_j\to W_j$ which maps every point in $\bar{W}_j$ to the equivalent point in $W_j$, are diffeomorphism. Moreover
\be
 {\psi_1}^* g_{\mathbb{S}^2}(V,V)=\psi_2^*g_j (V,V)\quad\forall V\in T\bar{W}_j.
 \ee
  Therefore  $\epsilon$  as in (\ref{Three}) and (\ref{Four}) can be chosen very small. We fix  $\epsilon=\frac{1}{j^2}$. 
We have 
\begin{eqnarray}
 D_{\bar{W}_j}&\leq& \diam(S^2)\leq \pi\\
 D_{W_j}&\leq& \diam(M_j)\leq 2\pi 
 \end{eqnarray}
  and  therefore 
 \be
a<\tfrac{\arccos(1+\epsilon)^{-1}}{\pi}\max\{D_{\bar{W}_j},D_{W_j}\}
=2\arccos(\tfrac{j^2}{j^2+1}).
\ee
 We fix 
\begin{eqnarray}\label{arccos}
a=2\arccos(\tfrac{j}{j+1}).
\end{eqnarray}
For $\lambda$, $h$ as defined in (\ref{lambda}) and  (\ref{h})  we have
\begin{eqnarray}
\lambda&=&\sup_{x,y\in \bar{W}_j}\left|d_{\mathbb{S}^2}(\psi_1(x),\psi_1(y))-d_{M_j}(\psi_2(x),\psi_2(y))\right|\\
&\leq& \sup_{x,y\in W_j}  [\min\{\tfrac{d(x,y)}{\epsilon^\prime_j}, 2j \}(\pi-2)\epsilon^\prime_j]\\
&\leq& {2 (\pi-2)}{j\epsilon^\prime_j}\nonumber
\end{eqnarray}
 and 
\begin{eqnarray}
h&=&\sqrt{\lambda(\max\{D_{\bar{W}_j},D_{W_j}\}+\tfrac{\lambda}{4})}\\
&\leq&\sqrt{{4\pi^2 (\pi-2)}\cdot{j\epsilon^\prime_j} +\tfrac{ \pi-2}{2}\cdot{j\epsilon^\prime_j}}.\nonumber
\end{eqnarray}

  Therefore 
 \be\label{barh}
\bar{h}=\max \{h,\sqrt{\epsilon^2+2\epsilon}D_{W_j},\sqrt{\epsilon^2+2\epsilon}D_{\bar{W}_j}\}\leq\max\{ \tfrac{4\pi}{j},h\}.
\ee
Moreover we have
\begin{eqnarray}\label{equations}
\area(\mathbb{S}^2\backslash \bar{W}_j)&=&2j\pi{\epsilon^\prime_j}^2\\
\nonumber\area(M_j\backslash W_j)&\leq&2\pi^2\epsilon^\prime_j\\
\nonumber\area(W_j)=\area(\bar{W}_j)&\leq& \pi\\
\nonumber L(\partial W_j)=L(\partial\bar{W}_j)&\leq& 2j\pi\epsilon^\prime_j.
\end{eqnarray}
Because  of the choice of $\epsilon^{\prime}_j$ as  in the beginning of the proof of this example, we have $j\epsilon^{\prime}_j\to 0$ as $j$ goes to $\infty$. 
Considering the bounds  (\ref{arccos}), (\ref{barh}) and  (\ref{equations}) and inequality (\ref{flatdistance}), we conclude 
\be
 d_{\mathcal{F}}(M_j,\mathbb{S}^2)\Fto 0.
 \ee

We claim  that {\em there exist $\delta_j$ in $\CovSpec(M_j)$  such that $\delta_j\to\delta_0>0$}.
All loops in $M_j$ are generated either by combinations of   loops which go around the cylinders of length $2\pi\epsilon_j^\prime$ with  loops which goes along the cylinders and $W_j$ of length 
\be
2\delta_j=\pi+d(p_{2i},p_{2i-1})-2\epsilon_j.
\ee
 For $\delta\geq\delta_j$, these loops lift closed to $M_j^{\delta}$ and so $M_j^{\delta}= M_j$.
For $\pi/2\leq \delta<\delta_j$ the loops along the cylinder lift open and loops around cylinder lift closed. Therefore  $M_j^{\delta}
=M_j^{\pi/2}$ as depicted in Figure~\ref{fig-Many more spheres}.  Thus $\delta_j\in \CovSpec(M_j)$.   However
\be
\delta_j \to \delta_0=\pi+d(p_{2i},p_{2i-1}) \notin \CovSpec(M_\infty)\cup \{0\}=\{0\}.
 \ee

Finally we claim that {\em the sequence $\tilde{M}_j^{\delta}$ for $\delta=\pi/2$
depicted in Figure~\ref{fig-Many more spheres}
has no subsequence converging in the pointed intrinsic flat sense}. 
We use a similar argument as in Example 9.1 in \cite{Sormani-AA}
of the second author.   

First observe that $\tilde{M}_j^\delta$ has infinite volume and so
it must be viewed as a locally integral current space in the sense of
Definition~\ref{defn-locally}.   Recall also the definition of pointed intrinsic flat convergence in Definition~\ref{defn-pointed-IF-conv}.

Suppose on the contrary that a subsequence, again denoted by $\tilde{M}_{j}^{\delta}$, converges in the
pointed intrinsic flat sense to some some locally integral current space we call $M_\infty^\delta$ possibly the $\bf{0}$ space.   
This means, there exists $\tilde{p}_j\in \tilde{M}_j^\delta$ 
and $\tilde{p}_\infty \in M_\infty^\delta$ such that 
\be
\forall R>0 \qquad
S(\tilde{p}_j,R) 
\Fto S(\tilde{p}_\infty,R).
\ee
We will in fact prove that $S(\tilde{p}_j, 3\pi)$ depicted
in Figure~\ref{fig-Many more spheres} has no intrinsic flat
limit because it contains increasingly many copies of
isometric balls.

Choose 
\be
\tilde{y}_j\in S(\tilde{p}_j, \pi)\subset \tilde{M}_{j}^{\delta}
\ee 
to be the unique lifts of points $y_j$ in the interior of $W_j$ such that  for $r$ small enough  $B(y_j,r)$ remains in the interior of $W_j $.    In this way $B(\tilde{y}_j, r)\subset
S(\tilde{p_j}, 3\pi)$ is isometric to $B(y_j, r)$ which is isometric
to a ball in a standard sphere.    Observe that such balls cannot 
disappear in the limit because they are all isometric to one another.
Thus, in particular $M^\delta_\infty \neq \bf{0}$.

We claim that {\em there are 
\be
x_{j,1},x_{j,2},\ldots,x_{j,2j}\in S(\tilde{p}_j, 3\pi)\subset \tilde{M}_{j}^{\delta}
\ee
distinct liftings of the point $y_j$ in $M_j$ distinct from $\tilde{y}_j$.}   Note that each pair $x_{j,2i-1}, x_{j, 2i}$ is found by lifting
a closed loop through the $i^{th}$ cylinder based at $y_j$ to
a path from $\tilde{y}_j$ to $x_{j, 2i-1}$ and by lifting same
the closed loop traversed in the opposite direction 
to a path from $\tilde{y}_j$ to $x_{j, 2i}$.   Since $M_j$ contains
$j$ cylinders, we obtain $2j$ such points as claimed.

Next observe that
\begin{eqnarray}
d_{\tilde{M}_j^\delta}(x_{j,i},x_{j,k}) > 2r \qquad \forall i,k \in \{1,2,\ldots,2j\}.
\end{eqnarray}
Thus
$
B(x_{j,k},r)
$
are disjoint and are all isometric to a ball $B(x,r)$ in a standard sphere.
Thus
\begin{eqnarray}
d_{\mathcal{F}}( S(x_{j,k},r), S(x, r))=0 \qquad \forall k \in \{1,2,\ldots,2j\}.
\end{eqnarray}
and
\begin{eqnarray}
d_{\mathcal{F}}( S(x_{j,k},r), {\bf{0}})=
h_0=d_{\mathcal{F}}( S(x, r), {\bf{0}})>0 ~~~ \forall k \in \{1,2,\ldots,2j\}.
\end{eqnarray}
Applying the Bolzano-Weierstrass Theorem of the second author proven in \cite{Sormani-AA} (cf. Theorem~\ref{B-W-BASIC}), there is a subsequence of each
$x_{j,k}$ which converge to some $x_k \in S(\tilde{p}_\infty, 3\pi) \subset M_\infty^\delta$.    Diagonalizing, there
is a subsequence again denoted by $x_{j,k}$ such that $x_{j,k}\to x_k$ for all $k$ such that 
\begin{eqnarray}
d_{ M_\infty^\delta}(x_k, x_{k'})> 2r
\end{eqnarray}
so that
$
B(x_{j,k},r)
$
are disjoint.  Applying (\ref{balls-converge}),
\begin{eqnarray}
\lim_{j\to\infty}d_{\mathcal{F}}( S(x_{j,k},r), S(x_k, r))=0 
\end{eqnarray}
and so
\begin{eqnarray}
d_{\mathcal{F}}( S(x_{k},r), S(x, r))=0. 
\end{eqnarray}
Thus $S(\tilde{p}_\infty, 3\pi)$ contains infinitely many balls of the same mass, which contradicts the
fact that
$\mass(S(\tilde{p}_\infty, 3\pi))$ is finite.  

So $\tilde{M}_j^\delta$
has no converging subsequence.

\end{proof}

\section{Open Questions}

\begin{quest}
Can one define a reasonable notion for a product of
integral current spaces and prove that if the spaces converge,
$M_j \Fto M_\infty$
and $N_j \Fto N_\infty$, then $M_j\times N_j\Fto M_\infty\times N_\infty$?   Naturally the notion of product must extend the notion of
isometric product already defined for Riemannian manifolds.   Keep
in mind that Ambrosio-Kirchheim do not define the notion of a 
product of integral currents except for an integral current times an
interval in \cite{AK}.
\end{quest}

\begin{quest}
Can one impose a condition on a geodesic to guarantee that
it does not disappear under intrinsic flat convergence?   Perhaps
one might require that there exists $r>0$ such that the
tubular neighborhood of radius r about a geodesic is known to
converge in the intrinsic flat sense to a nonzero integral current
space.   Recall that in \cite{Sormani-AA} the second author proved that points will not
disappear if their balls are know to have nonzero intrinsic flat limits.
\end{quest}

\begin{quest}
Can one define new spectra which capture part of the covering
spectrum but behave better under intrinsic flat convergence?  One possible approach to avoid the disappearance of elements in the covering spectrum would be to consider the work of Plaut and Wilkins \cite{Plaut-Wilkins}, in which elements of the covering spectra 
are found using $\epsilon$-homotopies.   If one requires these $\epsilon$ homotopies to be built from points with uniform conditions that guarantee the points don't disappear under intrinsic flat convergence, then one may be able to define a new spectra
which converge well under intrinsic flat convergence.   
\end{quest}

\bibliographystyle{alpha}
\bibliography{2013}

\end{document}